\documentclass[review,10pt]{elsarticle}
\usepackage{srcltx}
\usepackage{tabls}
\usepackage{placeins}
\usepackage{eurosym}
\usepackage{multirow}
\usepackage{mathtools}
\usepackage{amsmath}
\usepackage{amsfonts}
\usepackage{amssymb}
\usepackage{amsthm}
\usepackage{graphicx}
\usepackage{mathrsfs}
\usepackage{xcolor}
\usepackage{exscale}
\usepackage{latexsym}
\usepackage{multicol}
\usepackage[backref=page,colorlinks,plainpages=true,pdfpagelabels,
hypertexnames=true,colorlinks=true,pdfstartview=FitV,linkcolor=blue,
citecolor=red,urlcolor=black]{hyperref}
\PassOptionsToPackage{unicode}{hyperref}
\PassOptionsToPackage{naturalnames}{hyperref}
\usepackage{enumerate}
\usepackage[shortlabels]{enumitem}
\usepackage{bookmark}
\usepackage{wasysym}
\usepackage{esint}
\usepackage[ddmmyyyy]{datetime}
\usepackage[margin=2cm]{geometry}
\makeatletter
\g@addto@macro\normalsize{%
  \setlength\abovedisplayskip{2pt}
  \setlength\belowdisplayskip{2pt}
  \setlength\abovedisplayshortskip{4pt}
  \setlength\belowdisplayshortskip{4pt}
}
\numberwithin{equation}{section}
\everymath{\displaystyle}
\usepackage[capitalize,nameinlink]{cleveref}
\crefname{section}{Section}{Sections}
\crefname{subsection}{Subsection}{Subsections}
\crefname{condition}{Condition}{Conditions}
\crefname{hypothesis}{Hypothesis}{Hypothesis}
\crefname{assumption}{Assumption}{Assumptions}
\crefname{lemma}{Lemma}{Lemmas}
\crefname{claim}{Claim}{Claims}
\crefname{remark}{Remark}{Remarks}

\crefformat{equation}{\textup{#2(#1)#3}}
\crefrangeformat{equation}{\textup{#3(#1)#4--#5(#2)#6}}
\crefmultiformat{equation}{\textup{#2(#1)#3}}{ and \textup{#2(#1)#3}}
{, \textup{#2(#1)#3}}{, and \textup{#2(#1)#3}}
\crefrangemultiformat{equation}{\textup{#3(#1)#4--#5(#2)#6}}%
{ and \textup{#3(#1)#4--#5(#2)#6}}{, \textup{#3(#1)#4--#5(#2)#6}}%
{, and \textup{#3(#1)#4--#5(#2)#6}}

\Crefformat{equation}{#2Equation~\textup{(#1)}#3}
\Crefrangeformat{equation}{Equations~\textup{#3(#1)#4--#5(#2)#6}}
\Crefmultiformat{equation}{Equations~\textup{#2(#1)#3}}{ and \textup{#2(#1)#3}}
{, \textup{#2(#1)#3}}{, and \textup{#2(#1)#3}}
\Crefrangemultiformat{equation}{Equations~\textup{#3(#1)#4--#5(#2)#6}}%
{ and \textup{#3(#1)#4--#5(#2)#6}}{, \textup{#3(#1)#4--#5(#2)#6}}%
{, and \textup{#3(#1)#4--#5(#2)#6}}

\crefdefaultlabelformat{#2\textup{#1}#3}
\newtheorem{theorem}{Theorem}[section]
\newtheorem{lemma}[theorem]{Lemma}
\newtheorem{claim}[theorem]{Claim}

\newtheorem{proposition}[theorem]{Proposition}

\newtheorem{definition}[theorem]{Definition}% Use {\rm ...}
\newtheorem{remark}[theorem]{Remark}        % Use {\rm ...}
  
\newtheorem{hypothesis}[theorem]{Hypothesis}
\numberwithin{equation}{section}
\def\aa{\mathcal{A}}

\newcommand{\pa}{\partial}

\makeatletter
\newcommand{\vo}{\vec{o}\@ifnextchar{^}{\,}{}}
\makeatother
\def\Yint#1{\mathchoice
    {\YYint\displaystyle\textstyle{#1}}%
    {\YYint\textstyle\scriptstyle{#1}}%
    {\YYint\scriptstyle\scriptscriptstyle{#1}}%
    {\YYint\scriptscriptstyle\scriptscriptstyle{#1}}%
      \!\iint}
\def\YYint#1#2#3{{\setbox0=\hbox{$#1{#2#3}{\iint}$}
    \vcenter{\hbox{$#2#3$}}\kern-.50\wd0}}
\def\longdash{-\mkern-9.5mu-} %USE THIS IF \usepackage{fourier} IS NOT USED. 
\def\tiltlongdash{\rotatebox[origin=c]{18}{$\longdash$}}
\def\fiint{\Yint\tiltlongdash}
\def\Xint#1{\mathchoice
    {\XXint\displaystyle\textstyle{#1}}%
    {\XXint\textstyle\scriptstyle{#1}}%
    {\XXint\scriptstyle\scriptscriptstyle{#1}}%
    {\XXint\scriptscriptstyle\scriptscriptstyle{#1}}%
      \!\int}
\def\XXint#1#2#3{{\setbox0=\hbox{$#1{#2#3}{\int}$}
    \vcenter{\hbox{$#2#3$}}\kern-.50\wd0}}
\def\hlongdash{-\mkern-13.5mu-}
\def\tilthlongdash{\rotatebox[origin=c]{18}{$\hlongdash$}}
\def\hint{\Xint\tilthlongdash}
\makeatletter
\def\namedlabel#1#2{\begingroup
   \def\@currentlabel{#2}%
   \label{#1}\endgroup
}
\makeatother
\makeatletter
\newcommand{\rmh}[1]{\mathpalette{\raisem@th{#1}}}
\newcommand{\raisem@th}[3]{\hspace*{-1pt}\raisebox{#1}{$#2#3$}}
\makeatother

\newcommand{\lsb}[2]{#1_{\rmh{-3pt}{#2}}}
\newcommand{\lsbo}[2]{#1_{\rmh{-1pt}{#2}}}

\newcommand{\redref}[2]{\texorpdfstring{\protect\hyperlink{#1}{\textcolor{black}{(}\textcolor{red}{#2}\textcolor{black}{)}}}{}}
\newcommand{\redlabel}[2]{\hypertarget{#1}{\textcolor{black}{(}\textcolor{red}{#2}\textcolor{black}{)}}}
\newcommand{\descitem}[2]{\item[{(#1):}]\label{#2}}
\newcommand{\descref}[2]{\hyperref[#1]{\textnormal{\textcolor{black}{(}\textcolor{blue}{\bf #2}\textcolor{black}{)}}}}
\newcommand{\ditem}[2]{\item[#1] \label{#2}}
\newcommand{\dref}[2]{\hyperref[#1]{\textcolor{black}{(}\textcolor{blue}{\bf #2}\textcolor{black}{)}}}
\newcommand{\tp}{\tilde{p}}

\newcommand{\tw}{\tilde{w}}

\newcommand\RR{\mathbb{R}}

\newcommand\ZZ{\mathbb{Z}}
\newcommand\NN{\mathbb{N}}
\newcommand{\al}{\alpha}
\newcommand{\be}{\beta}

\newcommand{\de}{\delta}
\newcommand{\ve}{\varepsilon}

\newcommand{\ka}{\kappa}

\newcommand{\la}{\lambda}

\newcommand{\Om}{\Omega}

\DeclareMathOperator{\dv}{div}
\DeclareMathOperator{\spt}{spt}

\DeclareMathOperator{\loc}{loc}
\newcommand{\iprod}[2]{\langle #1 \ ,  #2\rangle}
\newcommand{\abs}[1]{\left| #1\right|}

\newcommand{\lbr}[1][(]{\left#1}
\newcommand{\rbr}[1][)]{\right#1}
\newcommand{\avgs}[2]{\lsbo{\lbr #1 \rbr}{#2}}

\newcommand{\txt}[1]{\qquad \text{#1} \qquad}
\newcounter{whitney}
\refstepcounter{whitney}

\newcounter{ineqcounter}
\refstepcounter{ineqcounter}

\makeatletter
\def\ps@pprintTitle{%
\let\@oddhead\@empty
\let\@evenhead\@empty
\def\@oddfoot{}%
\let\@evenfoot\@oddfoot}
\makeatother
\usepackage{setspace}
\usepackage[titletoc,toc,page]{appendix}

\nonstopmode
\begin{document}

\begin{frontmatter}

\title{$C^{1,\alpha}$ regularity for quasilinear parabolic equations with nonstandard growth}

\author[myaddress]{Karthik Adimurthi\tnoteref{thanksfirstauthor}}
\ead{karthikaditi@gmail.com and kadimurthi@tifrbng.res.in}

\author[myaddress]{Suchandan Ghosh\tnoteref{thankssecondauthor}}
\ead{suchandan@tifrbng.res.in}

\author[myaddress]{Vivek Tewary\tnoteref{thankssecondauthor}}
\ead{vivektewary@gmail.com and vivek2020@tifrbng.res.in}

\tnotetext[thanksfirstauthor]{Supported by the Department of Atomic Energy,  Government of India, under
	project no.  12-R\&D-TFR-5.01-0520 and SERB grant SRG/2020/000081}
\tnotetext[thankssecondauthor]{Supported by the Department of Atomic Energy,  Government of India, under
	project no.  12-R\&D-TFR-5.01-0520}

\address[myaddress]{Tata Institute of Fundamental Research, Centre for Applicable Mathematics,Bangalore, Karnataka, 560065, India}

\begin{abstract}
In this paper, we obtain $C^{1,\alpha}$ estimates for weak solutions of certain quasilinear parabolic equations satisfying nonstandard growth conditions, the prototype examples being
\[
  \left\{ \begin{array}{l}  u_t - \dv (|\nabla u|^{p-2} \nabla u + a(t)|\nabla u|^{q-2} \nabla u) = 0, \\
    u_t - \dv (|\nabla u|^{p(t)-2} \nabla u) = 0, \end{array} \right.
\]
under the assumption that the solutions a priori have bounded gradient. We build on the recently developed scaling and covering argument  which allows us to consider the singular and degenerate cases in a uniform manner and with minimal regularity requirements on the phase switching factor $a(t)$ and the variable exponent $p(t)$. Moreover, we are able to take any $p \leq q < \infty$ to obtain the desired regularity. 

\end{abstract}


\end{frontmatter}
\begin{singlespace}
\tableofcontents
\end{singlespace}
\section{Introduction}
\label{section1}
In this paper, we study gradient regularity of weak solutions of two quasilinear parabolic equations satisfying nonstandard growth conditions assuming that the solutions have bounded gradient. We discuss each of the problems in the following subsections:

\subsection{Multiphase Problems}
The first problem is the multiphase equation having the form:
\begin{align}\label{maineq1}
	u_t-\dv \left(\mathcal{A}_p(\nabla u)+\sum_{i=1}^k a_i(t)\mathcal{A}_{q_i}(\nabla u)\right)=0 \txt{ in } \mathcal{D}^{'}(\Omega_T),
\end{align} where $\mathcal{A}_p(\cdot)$ and $\mathcal{A}_{q_i}(\cdot)$ satisfy the following growth and coercivity conditions for $i\in\{1,2,\ldots,k\}$:
\begin{description}
    \descitem{H1}{H1} $\left\langle \mathcal{A}'_p(z)\zeta,\zeta\right\rangle\geq C_0 |z|^{p-2}|\zeta^2|$ and $\left\langle \mathcal{A}'_{q_i}(z)\zeta,\zeta\right\rangle\geq C_0 |z|^{q_i-2}|\zeta^2|$.
    \descitem{H2}{H2} $|\mathcal{A}_p(z)|+|\mathcal{A}'_p(z)||z| \leq C_1 |z|^{p-1}$ and $|\mathcal{A}_{q_i}(z)|+|\mathcal{A}'_{q_i}(z)||z| \leq C_1 |z|^{q_i-1}$.
    \descitem{H3}{H3} $0\leq a_i(t)\leq M$ is a bounded, measurable function for some fixed number $M$.
\end{description} 
Here we have denoted $\mathcal{A}'_{\cdot}(z)\coloneqq \frac{d\mathcal{A}_{\cdot}(z)}{dz}$. 

The study of elliptic multiphase problems was first studied in the papers \cite{zhikovLavrentievPhenomenonHomogenization1993,zhikovLavrentievPhenomenon1995} in homogenization theory where it serves as a model for highly anisotropic materials. A systematic study of their regularity properties was initiated in \cite{marcelliniRegularityMinimizersIntegrals1989}.

Since these equations are similar to the parabolic $p$-Laplace equation having the prototype form $u_t - \dv \aa_p(\nabla u)$, we follow the idea of intrinsic scaling developed by DiBenedetto-Friedman~\cite{dibenedettoRegularitySolutionsNonlinear1984,dibenedettoHolderEstimatesNonlinear1985ad} (see also \cite[Chapter IX]{dibenedettoDegenerateParabolicEquations1993} for more details).

Since the  technique of intrinsic scaling  was developed to handle two different scalings, whereas the equation \cref{maineq1} has more than two different scaling factors, it is not entirely clear how to adapt the techniques from \cite{dibenedettoHolderEstimatesNonlinear1985} to study \cref{maineq1}.  As a result, obtaining $C^{1,\alpha}$ regularity was a difficult problem even in the case $a_i(t) \equiv 1$ and $k=1$. These sort of equations have gained a lot of interest recently, see \cref{tab:my-table}, \cref{tab:my-table2} and \cref{tab:my-table3} for some of the recent developments regarding regularity for these equations. 

In this paper, we obtain $C^{1,\alpha}$ regularity for weak solutions of \cref{maineq1} by suitably adapting the arguments from \cite{dibenedettoHolderEstimatesNonlinear1985} and combining it with  with the scaling and covering arguments from \cite{adimurthiUnifiedApproachAlpha2020}. 
\subsubsection{Comparison to past results and improvements} We shall compare below our results with previous developments in this subject and highlight some of the important improvements:
\begin{itemize}
    \item The most important aspect of our approach is that we no not need any restrictions relating $p$ and  $\{q_i\}$. In the double phase elliptic case, $C^{1,\alpha}$ regularity  was proved with $q$ satisfying some additional assumptions, due to the presence of  Lavrentiev phenomenon. We are able to sidestep such considerations since we  start with the assumption that  the solution  is Lipschitz continuous and hence we are automatically in the absence of Lavrentiev phenomenon. 
    \item Due to the uniform nature of the proof, we do not need to distinguish between the singular and degenerate regimes,  in particular we impose no additional restrictions on  $p$ and $\{q_i\}$.
    \item We only assume non-negativity, boundedness and measurability on the coefficients $\{a_i(t)\}$. In the elliptic setting, $C^{1,\alpha}$ theory was developed in \cite{colomboRegularityDoublePhase2015,colomboBoundedMinimisersDouble2015} and the results of this paper paves the way for implementing the strategy from \cite{colomboRegularityDoublePhase2015,colomboBoundedMinimisersDouble2015} to obtain analogous results in the parabolic setting with additional assumptions satisfied by  $\{a_i(x,t)\}$  only in the $x$ variable.
\end{itemize}

\begin{remark}[Assumption of Lipschitz regularity] In this paper, we have assumed that the solutions a priori have bounded spatial gradient, i.e., $|\nabla u| \in L^{\infty}_{\loc}$. This allows us the freedom to not impose any restrictions on $p$ and $\{q_i\}$ which are customary in regularity results for multiphase problems. A brief survey of some well known regularity results for problems with $(p,q)$-growth is given in~\cref{tab:my-table},~\cref{tab:my-table2} and~\cref{tab:my-table3}, noting that a more detailed list can be found in the references therein.
\end{remark}

\begin{table}[ht!]
\caption{Local boundedness for problems with nonstandard growth}
	\resizebox{\textwidth}{!}{%
		\bgroup
		\def\arraystretch{2}
		\begin{tabular}{l|l|l|l|c}
			\hline\hline
			\textit{Equation or Variational Integral} &
			\textit{Assumption on $a$} &
			\textit{Assumption on $u$} &
			%\multicolumn{2}{l} 
			\textit{Restriction on $p,q$} & \textit{References}   \\ \hline
			\multirow{2}{0.2\linewidth}{\begin{tabular}[c]{@{}l@{}}$\int_{\Omega}F(x,u,\nabla u)dx$ where \\ $C_0|z|^p\leq F(x,u,\nabla u)\leq C_1(1+|z|^q)$\end{tabular}} &
			\multirow{2}{0.2\linewidth}{$a(x)$ is non-negative, measurable and bounded.} &
			\multirow{2}{0.2\linewidth}{$u$ is a minimizer of the functional of nonstandard growth.} &
			$q \leq p^* = \frac{Np}{N-p}$ & \cite{marcelliniRegularityMinimizersIntegrals1989,moscarielloHolderContinuityMinimizers1991,boccardoInftyRegularityVariational1990,fuscoLocalBoundednessMinimizers1990}
			\\ \cline{4-5} 
			&
			&
			&
			$q \leq p_{N-1}^* = \frac{(N-1)p}{N-1-p}$ & \cite{hirschGrowthConditionsRegularity2020a}
			\\ \hline
			\multirow{4}{*}{$\partial_t u - \mbox{div}\left( |\nabla u|^{p-2}\nabla u + a(x,t)|\nabla u|^{q-2}\nabla u \right) = 0$} &
			\multirow{4}{0.2\linewidth}{$a(x,t)$ is non-negative, measurable and bounded.} &
			\multirow{2}{*}{$u$ is a weak solution.} &
			\multirow{2}{*}{$q \leq p_* = p \frac{N+2}{N}$} & \cite{yuBoundednessSolutionsParabolic1997}
			\multirow{2}{*}{} \\
			&
			&
			&
			&
			\\ \cline{3-5} 
			&
			&
			$u$ is a variational solution.&
			\multirow{2}{*}{$q < p_* = p \frac{N+2}{N}$} & \cite{singerLocalBoundednessVariational2016}
			\multirow{2}{*}{} \\
			&
			&
			&
			&
			\\ \hline\hline
		\end{tabular}%
		\egroup
	}
	\label{tab:my-table}
\end{table}
For further results regarding boundedness of local minimisers, we refer to \cite{hirschGrowthConditionsRegularity2020a} for a detailed history, noting that \cite{marcelliniRegularityExistenceSolutions1991} contains sharp counterexamples. 

\begin{table}[ht!]
		\caption{Local boundedness of gradient for minimizers of autonomous integrals and solutions of related parabolic equations.}
	\bgroup
	\def\arraystretch{2.5}
	\resizebox{\textwidth}{!}{%
		\begin{tabular}{l|l|l|lc}
			\hline\hline
            \textit{Equation or Variational Integral} & \textit{Assumptions / Hypotheses}     & \textit{Conclusion}                        & \textit{References}                     \\ \hline
			\multirow{4}{0.5\linewidth}{
				$\begin{aligned}[t]
					&\int_{\Omega}F(\nabla u)dx  \ \mbox{ where } \\ 
					&\qquad z \to F(z) \mbox{ is } C^2, \\ 
					&\qquad \nu|z|^p\leq F(z)\leq L(1+|z|^q), \\ 
					&\qquad \langle F_{zz}\lambda,\lambda\rangle \geq \nu(1+|z|^2)^{\frac{p-2}{2}}|\lambda|^2,\\
					&\qquad \langle F_{zz}\lambda,\lambda\rangle\leq L(1+|z|^2)^{\frac{q-2}{2}}|\lambda|^2.
				\end{aligned}$
			} &
			$2 \leq p < q$, $ \frac{q}{p} < 1 + \frac{2}{N-2}$ with $N \geq 3$  &
			$u \in W^{1,q} \Rightarrow u \in W^{1,\infty}$ & \cite{marcelliniRegularityExistenceSolutions1991}
			\\ \cline{2-4} 
			& $2 \leq p < q$, $ \frac{q}{p} < 1 + \frac{2}{N}$ & $u \in W^{1,p} \Rightarrow u \in W^{1,\infty}$ & \cite{marcelliniRegularityExistenceSolutions1991}  \\ \cline{2-4}
			
			& $2 \leq p < q$, $ \frac{q}{p} < 1 + \frac{2}{N-3}$ with $N \geq 4$                             & $u \in W^{1,q} \Rightarrow u \in W^{1,\infty}$ & \cite{bellaRegularityMinimizersScalar2020} \\ \cline{2-4} 
			& $2 \leq p < q$, $ \frac{q}{p} < 1 + \min\left\{1,\frac{2}{N-1}\right\}$ with $N \geq 2$                             & $u \in W^{1,1} \Rightarrow u \in W^{1,\infty}$              & \cite{bellaRegularityMinimizersScalar2020} \\ \cline{2-4}
			& $2 \leq p < q$, $ \frac{q}{p} < 1 + 2\min\left\{\frac{1}{p},\frac{1}{N}\right\}$ with $N \geq 2$                             & $u \in W^{1,p} \Rightarrow u \in W^{1,q}$              & \cite{espositoHigherIntegrabilityMinimizers1999} \\ \hline
            \multirow{4}{*}{$\begin{aligned}[t]
					& \partial_t u - \mbox{div}\, b(t,x,\nabla u)   = 0 \\
					&\qquad \mbox{$b$ is differentiable in $x$ and $z$,} \\
					&\qquad|b(t,x,z)|+|\nabla_{z}b(t,x,z)| |z|\leq L|z|^{q-1},\\  &\qquad\langle\nabla_{z}b(t,x,z)\zeta,\zeta\rangle\geq\nu|z|^{p-2}|\zeta|^2,\\  
					&\qquad|\nabla_x b_j(t,x,z)|\leq L|z|^{p+q-2\over 2},\\  &\qquad|D_{z_j}b_i(t,x,z)-D_{z_i}b_j(t,x,z)|\leq L|z|^{p+q-4\over 2}.  \end{aligned}$} &
			$\begin{aligned}[t]
				&\frac{2N}{N+2}\leq p \leq q \leq p + \frac{4}{N},\\
				&\mbox{for example, } b(t,x,z)=|z|^p+a(t)|z|^q.
				\end{aligned}$   &    $u \in L^q(W^{1,q}) \Rightarrow \nabla u \in L^{\infty}$                   & \cite{bogeleinParabolicEquationsGrowth2013,singerParabolicEquationsGrowth2015}  \\
			     \cline{2-4} &
			     $\begin{aligned}[t]
			     	&\frac{2N}{N+2}\leq p \leq q \leq p + \frac{4}{N+2},\\
			     	&\mbox{for example, } b(t,x,z)=|z|^p+a(t)|z|^q.
			     \end{aligned}$   &        $\begin{aligned}
			     &\mbox{Existence of solution in }L^q(W^{1,q}),\\
			     &\nabla u\in L^\infty
		     \end{aligned}$              & \cite{bogeleinParabolicEquationsGrowth2013,singerParabolicEquationsGrowth2015}  \\
			     \cline{2-4} 
			& $\begin{aligned}[t]
				&\frac{2N}{N+2} < p< q < p + \frac{\min\{2,p\}}{N+2},\\
				& \mbox{for example, } b(t,x,z)=|z|^p+a(t,x)|z|^q. 
				\end{aligned}$ & Existence of solution in $L^q(W^{1,q}) \cap L^{\infty}$                      & \cite{singerExistence2016} \\ 
			\hline\hline
		\end{tabular}%
	}
	\egroup
	\label{tab:my-table2}
\end{table}
\FloatBarrier
We mention some other papers that include important contribution towards Lipschitz regularity which don't necessarily  fit into the framework of \cref{tab:my-table2}:  \cite{espositoRegularityResultsMinimizers2002,defilippisGradientBoundsSolutions2020}

\begin{table}[ht!]
\caption{Regularity results for minimizers of non-autonomous integrals with non-standard growth and solutions of related parabolic equations.}
% 	\bgroup
	\def\arraystretch{2}
	\resizebox{\textwidth}{!}{%
		\begin{tabular}{l|l|l|l|l}
			\hline\hline
			\textit{Equation or Variational Integral} &
            \textit{Hypotheses} &
            \textit{Restriction on $p,q$} &
			%\multicolumn{2}{l}
			\textit{Conclusion} & \textit{Reference} \\ \hline
			\multirow{5}{*}{$\begin{aligned}[t] &\int_\Omega F(x,\nabla u)\,dx, \\ & F(x,z) = |z|^p + a(x) |z|^q, \\ & 0 \leq a(x) \leq M. \end{aligned}$} &
			\multirow{2}{*}{$\begin{aligned}[t] & a\in C^{0,\alpha}(\Omega), \\ & 0<\al\leq 1.\end{aligned}$} &
			$\frac{q}{p}<1+\frac{\al}{N}$ &
			$u \in W^{1,p} \Rightarrow u\in W^{1,q}$ & \cite{espositoSharpRegularityFunctionals2004}
			\\ \cline{3-5} 
			&
			&
			$\frac{q}{p}<1+\frac{\al}{N}$ &
			$u \in W^{1,p} \Rightarrow u\in C^{1,\beta}_{{\loc}}$ & \cite{colomboRegularityDoublePhase2015}
			\\ \cline{2-5} 
			&
			$\begin{aligned}[t] & a\in C^{0,\al}(\Omega), \\ & 0<\al\leq 1,\\ & u \in L^\infty_{\loc}.\end{aligned}$ &
			$q<p+\al$ &
			$u \in W^{1,p} \Rightarrow u\in C^{1,\beta}_{\loc}$ & \cite{colomboBoundedMinimisersDouble2015}
			\\ \cline{2-5} 
			&
			\multirow{2}{*}{$\begin{aligned}[t] & a\in W^{1,d}(\Omega), \\ & d>N.\end{aligned}$} &
			\multirow{2}{*}{$\frac{q}{p}<1+\frac{1}{N}-\frac{1}{d}$} &
			\multirow{2}{*}{$\nabla u \in L^\infty_{\loc}$} & 
			\multirow{2}{*}{\cite{defilippisLipschitzBoundsNonautonomous2021}} \\
			&
			&
			&
			&
			\\ \hline
			$\begin{aligned}[t] & \int_\Omega F(x,\nabla u)\,dx, \\ &F(x,z) = |z|^p + a(x) |z|^q+c(x) |z|^s.\end{aligned}$ &
			$\begin{aligned}[t] & a\in C^{0,\beta_1}(\Omega),\\ & c\in C^{0,\beta_2}(\Omega), \\ &0<\beta_1,\beta_2\leq 1. \end{aligned}$ &
			$\begin{aligned}[t] & \frac{q}{p}<1+\frac{\beta_1}{N},\\ & \frac{s}{p}<1+\frac{\beta_2}{N}.\end{aligned}$ &
			$u\in C^{1,\alpha}_{\loc}$ & \cite{defilippisRegularityMultiphaseVariational2019}
			\\ \hline
			$\partial_t u - \mbox{div}\left( |\nabla u|^{p-2}\nabla u + a(x,t)|\nabla u|^{q-2}\nabla u \right) = 0$ &
			$\begin{aligned}[t] & a\in L^\infty(\Omega_T), \\ & \partial_x a \in L^d(\Omega_T),\\ & d>N. \end{aligned}$ &
			$\begin{aligned}[t]
				& q < p + 2\left(\frac{1}{N+2}-\frac{p}{2d}\right),\\
				& p > \frac{2Nd}{(N+2)(d-2)}.\end{aligned}$ &
			$\nabla u \in L^\infty_{\loc}$ & \cite{defilippisGradientBoundsSolutions2020}
			\\ \hline\hline
		\end{tabular}%
	}
% 	\egroup
	\label{tab:my-table3}
\end{table}

\subsection{Variable Exponent Problems}
The second problem we  study is the variable exponent problems having the form:
\begin{align}\label{maineq2}
	u_t-\dv (|\nabla u|^{p(t)-2}\nabla u)=0 \txt{ in } \mathcal{D}^{'}(\Omega_T),
\end{align} where the variable exponent $1<p\leq p(t)\leq q<\infty$ is a measurable function depending only on the time variable. Previously, $C^{1,\alpha}$ regularity was proved in \cite{bogeleinHolderEstimatesParabolic2012} for the more general exponent of the form $p(x,t)$ under the additional assumption on  $p(x,t)$ (H\"older continuous in $x$ and $t$). It is interesting to note that if we restrict to the case of $p(t)$, the approach in \cite{bogeleinHolderEstimatesParabolic2012} can be suitably modified to prove H\"older continuity of the gradient with $p(t)$ satisfying a much weaker $\log$-H\"older criterion, as obtained in~\cite{okRegularityParabolicEquations2018}. 

In this paper, we prove $C^{1,\alpha}$ regularity with minimal regularity assumptions on $p(t)$ and list below some of the important improvements:
\begin{itemize}
    \item In our approach, we only need to assume $p(t)$ is measurable to obtain gradient H\"older regularity.  Note that all previously known  approaches  required at least $\log$-H\"older continuity of $p(t)$ (see~\cite{okRegularityParabolicEquations2018}), thus our structural assumption is optimal and considerably weakens previously required hypothesis.
    \item Due to the uniform nature of our approach, we do not need to differentiate between the cases $p(t) \leq 2$ or $p(t) \geq 2$ throughout the proof.
    \item Our approach is  somewhat simpler  compared to \cite{bogeleinHolderEstimatesParabolic2012,okRegularityParabolicEquations2018} where a lot of variable exponent machinery was needed to be setup regarding the size of the cylinders and the exponent. On the other hand, we do not need any of those tools and instead, our proof follows from simpler considerations, all thanks to the additional freedom obtained through the new covering argument developed in~\cite{adimurthiUnifiedApproachAlpha2020}.
\end{itemize}

% We note that the standard technique of proving $C^{1,\alpha}$ regularity for quasilinear parabolic equations required using intrinsic geometry~\cite{dibenedettoHolderEstimatesNonlinear1985} and studying singular and degenerate cases separately. \emph{In this paper, we employ the newly developed scaling and covering argument from~\cite{adimurthiUnifiedApproachAlpha2020} which allows a uniform approach to $C^{1,\alpha}$ regularity without distinguishing between the singular and degenerate cases. It is this uniform approach which allows the treatment of these nonstandard growth problems.} 

\subsection{Approximation hypothesis}
\label{app_hyp}
In order to obtain $C^{1,\alpha}$ regularity, the strategy of the proof requires regularizing the equation and then passing through the limit in appropriate sense, for example, see \cite[Chapter VIII]{dibenedettoDegenerateParabolicEquations1993} for the details. 
\begin{hypothesis}\label{hyp_H}
We say that a divergence form operator $\mathcal{L}$  satisfies \cref{hyp_H}  if solutions to $\mathcal{L}u=0$ can be approximated  by smooth  functions $\{u^{\epsilon}\}$ which are solutions to $\mathcal{L}_{\epsilon} u^{\epsilon}=0$ where $\mathcal{L}_{\epsilon}$ satisfies similar structural assumptions as $\mathcal{L}$. 
\end{hypothesis}

We note that such approximations are known to exist for the prototype parabolic p-Laplacian problem and the variable exponent problem \cref{maineq2}.  These  approximations are crucially used to rigorously justify our proofs of the two alternatives \cref{alt1} and \cref{alt2} in the proof of $C^{1, \alpha}$ regularity. We however skip reference  to such approximations in our proofs because this is quite standard in the literature, see for instance \cite{kuusiWolffGradientBound2014}.

However for \cref{maineq1}, we are not aware whether \cref{hyp_H} holds with no restriction between $p$ and $\{q_i\}$. Hence, we assume such an hypothesis  holds true for \cref{maineq1}, and obtain $C^{1,\alpha}$ regularity  in \cref{corthem2}.  One scenario where our result from \cref{corthem2}  is applicable, is when there are additional restrictions on $\{q_i\}$ as obtained in \cite{defilippisGradientBoundsSolutions2020} where the regularization scheme is developed, see~\cref{tab:my-table3}. In particular, once \cref{hyp_H} is satisfied with Lipschitz regularity, then our approach immediately gives $C^{1,\alpha}$ regularity.

\section{Preliminaries}
\label{section2}

In this section, we shall collect all the preliminary material needed in subsequent sections.  Before we recall some useful results, let us define the notion of solutions considered in this paper. In order to do this, let us first define Steklov average as follows: let $h \in (0,2T)$ be any positive number, then we define
\begin{equation*}%\label{stek1}
  u_{h}(\cdot,t) := \left\{ \begin{array}{ll}
                              \hint_t^{t+h} u(\cdot, \tau) \ d\tau \quad & t\in (-T,T-h), \\
                              0 & \text{else}.
                             \end{array}\right.
 \end{equation*}
We shall now define the notion of weak solutions to \cref{maineq1} and \cref{maineq2}.
\begin{definition}[Weak solution of~\cref{maineq1}]
\label{weak_sol1}
    We say that \[u \in C^0(-T,T;L^2_{\loc}(\Om)) \cap L^p(-T,T;W^{1,p}_{\loc}(\Om))\cap \lbr \bigcap_{i=1}^k L^{q_i}(-T,T;W^{1,q_i}_{\loc}(\Om)) \rbr,\] is a weak solution of \cref{maineq1} if,  for any $\phi \in C_c^{\infty}(\Om)$ and any $t \in (-T,T)$,  the following holds:
\begin{equation*}
% \label{def_weak_solution}
  \int_{\Om \times \{t\}} \left\{ \frac{d [u]_{h}}{dt} \phi + \iprod{[\aa_p(\nabla u)]_{h}}{\nabla \phi} + \sum_{i=1}^k\iprod{[a_i(t)\aa_{q_i}(\nabla u)]_{h}}{\nabla \phi}\right\} \,dx = 0 \txt{for any}0 < t < T-h.
\end{equation*}
Analogously, we say that \[u \in C^0(-T,T;L^2_{\loc}(\Om)) \cap L^{p(\cdot)}(-T,T;W^{1,p(\cdot)}_{\loc}(\Om)),\] is a weak solution of \cref{maineq2} if,  for any $\phi \in C_c^{\infty}(\Om)$ and any $t \in (-T,T)$,  the following holds:
\begin{equation*}
% \label{def_weak_solution}
  \int_{\Om \times \{t\}} \left\{ \frac{d [u]_{h}}{dt} \phi + \iprod{[|\nabla u|^{p(\cdot)-2}\nabla u]_{h}}{\nabla \phi}\right\} \,dx = 0 \txt{for any}0 < t < T-h.
\end{equation*}
\end{definition}

\begin{definition}[Function Space]\label{func_space}
For any $1<\tp<\infty$ and any $m > 1$, we define the following Banach spaces:
\begin{equation*}%
\begin{array}{c}
V^{m,\tp}(\Om_T) := L^{\infty}(-T,T;L^m(\Om)) \cap L^{\tp}(-T,T;W^{1,\tp}(\Om)),\\
V^{m,\tp}_0(\Om_T) := L^{\infty}(-T,T;L^m(\Om)) \cap L^{\tp}(-T,T;W^{1,\tp}_0(\Om)).
\end{array}\end{equation*}%
These function spaces have the norm 
\begin{equation*}%\label{func_space_norm}
    \|f\|_{V^{m,\tp}(\Om_T)} := \sup_{-T < t<T} \|f(\cdot,t)\|_{L^m(\Om)} + \|\nabla f\|_{L^{\tp}(\Om_T)}.
\end{equation*}%
\end{definition}

We need the following parabolic Sobolev embedding, see \cite[Corollary 3.1 of Chapter I]{dibenedettoDegenerateParabolicEquations1993} for the details.
    \begin{lemma}\label{sobolev-poincare}
    Let $1<s<\infty$ and  $v \in V_0^{s}(Q)$ in some cylinder $Q=B \times I$, then
    \begin{equation*}
    \|v \|_{L^{s}(Q)}^{s} \leq C \abs{\{ |v| > 0\}}^{\frac{s}{N+s}} \|v \|_{V^{s}(Q)}^{s}.
    \end{equation*}
    \end{lemma}

Next we recall a well known iteration lemma, see \cite[Lemma 4.1 of Chapter I]{dibenedettoDegenerateParabolicEquations1993} for the details.
\begin{lemma}
    \label{iteration}
    Let $\{X_n\}$ for  $n=0,1,2,\ldots,$ be a sequence of positive numbers, satisfying the recursive inequalities
    \begin {equation*}
    X_{n+1} \leq C b^{n}  X_{n}^{1+\alpha},
    \end {equation*}
    where $C,b >1$ and $\alpha>0$ are given numbers. If
    \[
    X_0\leq C^{-\frac{1}{\alpha}}b^{-\frac{1}{\alpha^2}},
    \]
    then $\{X_n\}$ converges to zero as $n\rightarrow \infty$.
\end{lemma}

\subsection{Notation}
We list below the notation that we will use throughout the paper: 
\begin{enumerate}[(i)]
    
 \item\label{not1} We shall denote a point in $\RR^{N+1}$ by $z = (x,t) \in \RR^N \times \RR$.
 
 \item We shall use the notation $Q_{a,b}(x_0,t_0)$ to denote a parabolic cylinder of the form $B_a(x_0) \times (t_0-b,t_0+b)$. 
 \item\label{not11} Henceforth, we shall fix a cylinder $Q_0 = B_{R_0} \times (-R_0^2,R_0^2)$ centered at $(0,0)$ and its scaled version $4Q_0$.
  \item\label{not12} We shall denote the boundary of $4Q_0$ by $$\Gamma  = \lbr[[]B_{4R_0} \times\left\{t=-(4R_0)^2\right\}\rbr[]] \bigcup\lbr[[] B_{4R_0} \times\left\{t=(4R_0)^2\right\} \rbr[]]\bigcup \lbr[[]\pa B_{4R_0} \times \lbr-(4R_0)^2,(4R_0)^2\rbr\rbr[]].$$
 
 \item\label{not2} Let $\rho >0$, $\la \geq 1$ and $R_0 >0$ be fixed numbers, then for a given point $z_0 = (x_0,t_0) \in \RR^{N+1}$, we define the following cylinders:
 \begin{equation*}
% \begin{array}{rcl}
Q_{\rho} (x_0,t_0)  :=  B_{\rho}(x_0) \times (t_0-\rho^2, t_0+\rho^2) \txt{and}
Q_{\rho}^{\la} (x_0,t_0)  :=  B_{\la^{-1}\rho}(x_0) \times (t_0-\la^{-p}\rho^2, t_0+\la^{-p}\rho^2).
%Q_{\rho}^{(\la)} & := & Q_{\rho}^{(\la)}(0,0), \\
% \end{array}
\end{equation*}

\item\label{not3} Let $\la \geq 1$ be given, then for given two points $z_1 = (x_1,t_1) \in \RR^{N+1}$ and $z_2 = (x_2,t_2) \in \RR^{N+1}$, we need the following metrics:
 \begin{equation*}
	\begin{array}{ll}
		d (z_1, z_2)  := \max\{ |x_1 - x_2|,|t_1 - t_2|^{1/2}\}, & d_{\la} (z_1, z_2)  :=  \max\{\la|x_1 - x_2|,\la^{p/2}|t_1 - t_2|^{1/2}\}, \\
		d (z_1, \mathcal{K}) :=  \inf_{z_2 \in \mathcal{K}} d (z_1, z_2), & d_{\la} (z_1, \mathcal{K}) :=  \inf_{z_2 \in \mathcal{K}} d_{\la} (z_1, z_2). 
	\end{array}
\end{equation*}

 \item\label{not_par_bnd} For a given space-time cylinder $Q = B_R \times (a,b)$, we denote the parabolic boundary of $Q$ to be the union of the bottom and the lateral boundaries, i.e., $\pa_pQ = B_r \times \{t=a\} \bigcup \pa B_R \times (a,b)$.
 
\end{enumerate}
\section{Main Theorems}
\label{section3}

We obtain the following result for multiphase problems.
\begin{theorem}
\label{corthem1}
    Let $1 < p \leq \{q_1,q_2,\ldots,q_k\}< \infty$  for some $k \in \NN$ and $u$ be a weak solution of the prototype equation
    \[
        u_t - \dv \lbr \mathcal{A}_p(\nabla u) + \sum_{i=1}^k a_i(t) \mathcal{A}_{q_i}(\nabla u)\rbr = 0,
    \]
where $0 \leq a_i(t) \leq M$ is bounded, measurable functions and $\mathcal{A}_p$ and $\mathcal{A}_{q_i}$ satisfy the growth conditions in~\descref{H1}{H1},~\descref{H2}{H2} for $i\in\{1,2,\ldots,k\}$. Furthermore, assume that \cref{hyp_H}  and   $|\nabla u| \in L^{\infty}_{\loc}$ holds, then given any cylinder $Q_0 = B_{R_0} \times (-R_0^2,R_0^2)$,  there exists $\al= \al(N,p,\{q_i\},C_0,C_1,M,\mu_0) \in (0,1)$ such that for any $z_0, z_1 \in Q_0$, there holds
    \[
        |\nabla u(z_0) - \nabla u(z_1)| \leq C \mu_0^a \lbr \frac{d(z_0,z_1)}{R_0}\rbr^{\al},
    \]
    where $C= C(N,p,\{q_i\},C_0,C_1,M,\mu_0)$, $a= a(N,p,\{q_i\},\al)$ and 
    \[
       \mu_0 :=\sup_{4Q_0} |\nabla u|.
    \]
    Moreover, if $\mu_0\leq 1$, then $\al$ and $C$ can be taken to be independent of $\mu_0$. 
\end{theorem}

Next, we obtain the following regularity result for variable exponent problems.
\begin{theorem}
\label{corthem2}
    Let $1 < p \leq p(t) \leq q < \infty$ be a measurable variable exponent depending only on time variable and $u$ be a weak solution of the prototype equation
    \[
        u_t - \dv( |\nabla u|^{p(t) -2} \nabla u) = 0.
    \]
Furthermore, let us assume  $|\nabla u| \in L^{\infty}_{\loc}$ holds, then given any cylinder $Q_0 = B_{R_0} \times (-R_0^2,R_0^2)$,  we can find  $\al= \al(N,p,q,\mu_0) \in (0,1)$ such that for any $z_0, z_1 \in Q_0$, there holds
    \[
        |\nabla u(z_0) - \nabla u(z_1)| \leq C \mu_0^a \lbr \frac{d(z_0,z_1)}{R_0}\rbr^{\al},
    \]
    where $C= C(N,p,q,\mu_0)$, $a= a(N,p,q,\al)$ and 
    \[
        \mu_0 :=\sup_{4Q_0} |\nabla u|.
        \]
        Moreover, if $\mu_0\leq 1$, then $\al$ and $C$ can be taken to be independent of $\mu_0$. 
\end{theorem}

\begin{remark}
    Observe that in~\cref{corthem1} and~\cref{corthem2}, the H\"older exponent $\alpha$ and the constant $C$ depends on the gradient bound $\mu_0$ which   differs from~\cite{dibenedettoHolderEstimatesNonlinear1985,adimurthiUnifiedApproachAlpha2020} where these constants are independent of $\mu_0$. However, we note that the H\"older exponent $\alpha$ and the constant $C$ does not depend on the gradient bound $\mu_0$ provided that $\mu_0\leq 1$. Indeed, our proof of the $C^{1,\alpha}$ estimates proceed by assuming the condition $\mu_0\leq 1$ and to extend to the case $\mu_0 > 1$, we divide the solution by $\mu_0$ which changes the ellipticity constants for the nonlinear operators by a factor depending on $\mu_0$, which explains the dependence of the H\"older exponents and other constants on $\mu_0$.
\end{remark}

\section{Proof of \texorpdfstring{\cref{corthem1}}. for $\mu_0\leq 1$.}
\label{section6}

We shall provide a detailed proof for~\cref{corthem1} and indicate the modifications required in the proof of~\cref{corthem2}. The proofs of the two theorems differ only in the proofs of the two alternatives, namely,~\cref{alt1} and~\cref{alt2}, whereas the scaling and the covering arguments coincide. Further, it suffices to consider a double phase problem instead of the multiphase problem since the methods are the same.

 With $Q_0 = B_{R_0} \times (-R_0^2,R_0^2)$, let us consider equations of the form 
\begin{equation}
    \label{main_holder}
    u_t - \dv \lbr \mathcal{A}_p(\nabla u) +  a(t) \mathcal{A}_q(\nabla u)\rbr = 0,
\end{equation}
with $\aa_p(\zeta)$ and $\aa_q(\zeta)$ satisfying the following structural assumptions:
\begin{align}
    \left\langle \mathcal{A}'_p(z)\zeta,\zeta\right\rangle\geq C_0 |z|^{p-2}|\zeta^2| &\mbox{ and } \left\langle \mathcal{A}'_{q}(z)\zeta,\zeta\right\rangle\geq C_0 |z|^{q-2}|\zeta^2|.\\
    |\mathcal{A}_p(z)|+|\mathcal{A}'_p(z)||z| \leq C_1 |z|^{p-1} & \mbox{ and }  |\mathcal{A}_{q}(z)|+|\mathcal{A}'_{q}(z)||z| \leq C_1 |z|^{q-1},
\end{align} and $0\leq a(t)\leq M$ is a measurable function for some fixed number $M$

We have denoted $\aa'_{\cdot}(z) := \frac{d\aa_{\cdot}(z)}{dz}$. Let us fix the following constant:
\begin{equation}\label{sup_u}
    \mu_0:= \sup_{4Q_0} |\nabla u|,
\end{equation}
then for any $z \in Q_0$, we consider the cylinder $Q_S^{\mu_0}(z) =B_{\mu_0^{-1} S}(x) \times (t - \mu_0^{-p}S^2, t + \mu_0^{-p} S^2)$ to be the largest cylinder such that $Q_S^{\mu_0}(z) \subset 4Q_0$ and $Q_S^{\mu_0}(z) \cap (4Q_0 \setminus 2Q_0) \neq \emptyset$. Note that this fixes the radius $S$ and is independent of  the point $z \in Q_0$.  As a consequence, we have the following observations:
\begin{description}
        \descitem{O1}{obs1} Since $Q_S^{\mu_0}(z) \subset 4Q_0$, we see that  $S \leq \min\{\mu_0,\mu_0^{p/2}\}3R_0$ must hold. 
    \descitem{O2}{obs2} Moreover, since $Q_S^{\mu_0}(z)$ has to go outside $2Q_0$, we note that $S \geq  \min\{ \mu_0, \mu_0^{p/2}\} R_0$.
\end{description}

 Since the proof will be independent of the point $z \in Q_0$, we shall ignore writing the location of the cylinder $Q_R^{\mu}$. The proof of gradient H\"older regularity requires two propositions which are given below. We shall assume throughout this section that $\mu_0\leq 1$. Later, we will cover the case $\mu_0>1$.
\begin{proposition}\label{alt1}
    For some $\mu \leq \mu_0\leq 1$ and $R \leq S$, there exists numbers $\nu \in (0,1/2)$ and $\kappa,\de \in (0,1)$ depending only on $(N,p,C_0,C_1,M)$ such that if 
    \begin{equation*}%
        \abs{\{z \in Q_R^{\mu} : |\nabla u(z)| < \mu/2\}}< \nu |Q_R^{\mu}| \txt{and}  \sup_{Q_R^{\mu}} |\nabla u| \leq \mu,
    \end{equation*}%
    is satisfied, then the following conclusion follows:
    \begin{equation*}%
        \iint_{Q_{\de^{i+1}R}^{\mu}} |\nabla u - \avgs{\nabla u}{Q_{\de^{i+1}R}^{\mu}}|^2 \ dz \leq \ka \de^{N+2} \iint_{Q_{\de^{i}R}^{\mu}} |\nabla u - \avgs{\nabla u}{Q_{\de^{i}R}^{\mu}}|^2 \ dz
    \end{equation*}%
    for all $i \in \mathbb{Z}$.
\end{proposition}
\begin{proposition}\label{alt2}
   For some $\mu \leq \mu_0\leq 1$ and $R \leq S$ with $\nu$ as fixed in \cref{alt1}, there exists numbers $\frac12 < \sigma < \eta < 1$ such that if 
    \begin{equation*}%
        \abs{\{z \in Q_R^{\mu} : |\nabla u(z)| < \mu/2\}}\geq \nu |Q_R^{\mu}|,
    \end{equation*}%
    is satisfied, then the following conclusion follows:
    \begin{equation*}%
        |\nabla u(z)| \leq \eta \mu \qquad \text{for all}\  z \in Q_{\sigma R}^{\mu}.
    \end{equation*}%
\end{proposition}

\begin{remark}
    We remark that once the two alternatives above are established, the scaling argument and the proof of H\"older continuity follow verbatim from~\cite{adimurthiUnifiedApproachAlpha2020} and we refer to~\cite{adimurthiUnifiedApproachAlpha2020} for details on the scaling and covering argument.
\end{remark}

\subsection{Covering argument}
  Let us first fix some notation needed for the proof:
\begin{definition}\label{iter_const_def}
   Let $\mu$ be as in \cref{sup_u} and $\eta$, $\sigma$ be as in \cref{alt2}. Then let us denote
    \begin{equation*}
        \mu = \mu_0, \qquad \mu_{n+1} := \eta \mu_n,   \qquad R := S, \qquad R_{n+1} := c_0  R_n = c_0^{n+1} R,
    \end{equation*}
    where $c_0 := \frac12\sigma \min\{{\eta}, {\eta}^{p/2}\}$. Here we note that $c_0 \in (0,1)$ since $\sigma,\eta \in (0,1)$.
\end{definition}
We have the following results regarding the intrinsic geometry whose proofs may be found in~\cite{adimurthiUnifiedApproachAlpha2020}.
\begin{claim}\label{cyl_incl}
    With the notation as in \cref{iter_const_def}, we have the inclusion $Q_{R_1}^{\mu_1} \subset {Q_{\sigma R}^{\mu }}$.
\end{claim}

\begin{claim}\label{claim2.6}
    We have $R_n = c_0^n R$ and $\eta^n = \lbr \frac{R_n}{R} \rbr^{\al_1}$ where $\al_1^{-1} = - \log_{\frac{1}{\eta}} c_0$.
\end{claim}

\subsubsection{Switching radius}
Since we have to study the interplay between \cref{alt1} and \cref{alt2}, we need to define what is known as the switching radius:
\begin{definition}[Switching Radius]\label{switch_rad}
    With the notation from \cref{iter_const_def}, suppose the hypothesis of \cref{alt2} holds at level $(R_1,\mu_1),(R_2,\mu_2) \ldots (R_i,\mu_i)$, i.e., 
\begin{equation}\label{8.6}
    |\{ Q_{R_i}^{\mu_i} : |\nabla u| < \mu_i/2\}| \geq \nu |Q_{R_i}^{\mu_i}|,
\end{equation}
holds, then applying \cref{alt2}, we conclude
\begin{equation*}%
    |\nabla u| \leq \eta \mu_{i} = \mu_{i+1} \txt{on} Q_{\sigma R_{i}}^{\mu_{i}} \overset{\text{\cref{cyl_incl}}}{\supseteq} Q_{R_{i+1}}^{\mu_{i+1}}.
\end{equation*}%
Continuing this way, we denote  $n_0$ (called the switching number) and radius $R_{n_0}$ (called the switching radius) to be the first instance the condition from  \cref{8.6} fails, i.e., 
\begin{equation}\label{8.8eqn}
    |\{ Q_{R_{n_0}}^{\mu_{n_0}} : |\nabla u| < \mu_{n_0}/2\}| < \nu |Q_{R_{n_0}}^{\mu_{n_0}}|.
\end{equation}
\end{definition}

From \cref{switch_rad}, we see that the following estimates hold:
\begin{itemize}
\item For $n= 1,\ldots,n_0$, due to \cref{alt2} and Claim \ref{cyl_incl} there holds\begin{equation}\label{2.12}
    \sup_{Q_{R_n}^{\mu_n}} |\nabla u| \leq \mu_n = \eta^n \mu.
\end{equation}
\item We also have for any $n = 0,1,\ldots, n_0$, 
\begin{equation*}
  \sup_{Q_{R_{n}}^{\mu_n}}|\nabla u| \leq \mu \lbr \frac{R_n}{R}\rbr^{\al_1}.\end{equation*}

    This follows from \cref{2.12} and \cref{claim2.6}.

\item Since the first condition from  \cref{8.8eqn} is  the hypothesis required in  \cref{alt1},  there holds
\begin{equation}\label{8.10}
    \fiint_{Q_{\de^iR_{n_0}}^{\mu_{n_0}}} |\nabla u - (\nabla u)_i|^2 \ dz\leq \ka^i \fiint_{Q_{R_{n_0}}^{\mu_{n_0}}} |\nabla u - (\nabla u)_0|^2 \ dz \overset{\redlabel{8.10a}{a}}{\leq} \ka^i \mu_{n_0}^2 \txt{for} i=1,2,\ldots,
\end{equation}
where to obtain \redref{8.10a}{a}, we have used  the notation $(\nabla u)_i := \fiint_{Q_{\de^iR_{n_0}}^{\mu_{n_0}}} \nabla u \ dz$ and the bound   $ \sup_{Q_{R_{n_0}}^{\mu_{n_0}}} |\nabla u| \leq \mu_{n_0}$, which  holds due to \cref{2.12}. Moreover, we made use of  $\fiint |f - (f)|^2 \ dz = \inf_{a \in \RR} \fiint |f-a|^2\ dz \leq \fiint |f|^2$.
\end{itemize}

\begin{remark}
    From triangle inequality, for $i = 1,2,\ldots$, we also have 
    \begin{equation*}
    \begin{array}{rcl}
    \fiint_{Q_{\de^{i-1}R_{n_0}}^{\mu_{n_0}}} |\nabla u - (\nabla u)_i|^2 \ dz & \leq & 2\fiint_{Q_{\de^{i-1}R_{n_0}}^{\mu_{n_0}}} \hspace{-0.5cm}|\nabla u - (\nabla u)_{i-1}|^2 \ dz + 2\frac{|{Q_{\de^{i-1}R_{n_0}}^{\mu_{n_0}}}|}{|{Q_{\de^{i}R_{n_0}}^{\mu_{n_0}}}|}\fiint_{Q_{\de^{i-1}R_{n_0}}^{\mu_{n_0}}} \hspace{-0.5cm}|\nabla u - (\nabla u)_{i-1}|^2 \ dz\\
    & \overset{\cref{8.10}}{\leq} &  2\ka^{i-1} \lbr 1 + \frac{1}{\de^{N+2}} \rbr.
    \end{array}
\end{equation*}

\end{remark}

\begin{lemma}\label{lemma8.8}
Let $\kappa, \de, \eta$ be as given in \cref{alt1} and \cref{alt2}, then we have the following important consequences from \cref{2.12,8.10} whose proof may be found in~\cite{adimurthiUnifiedApproachAlpha2020}.
\begin{description}
    \descitem{C1}{conc1}  The sequence $\{(\nabla u)_i\}_{i=1}^{\infty}$ from \cref{8.10} is a Cauchy sequence and converges to $\nabla u(z_0)$ where  $z_0=(x_0,t_0)$ is the center of the parabolic cylinders considered in \cref{2.12} and \cref{8.10}.
    \descitem{C2}{conc2} The following decay estimate holds:
    \begin{equation*}%
     |\nabla u(x_0,t_0) - (\nabla u)_i|^2 \leq C \ka^i \mu_{n_0}^2 \txt{for all} i = 1,2\ldots.
 \end{equation*}%
    \descitem{C3}{conc3}  For any $0<\rho<R_{n_0}$ with $(\nabla u)_{\rho} := \fiint_{Q_{\rho}^{\mu_{n_0}}(x_0,t_0)} \nabla u \ dz$, the following decay estimate holds:
    \begin{equation*}%
     |\nabla u(x_0,t_0) - (\nabla u)_{\rho}|^2 \leq C(\de) \ka^{i} \mu_{n_0}^2,
 \end{equation*}%
 where $i \in \ZZ$ is such that $\de^i R_{n_0} \leq \rho \leq \de^{i-1} R_{n_0}$.
    \descitem{C4}{conc4} Let us define $\al_3 := \min \{ \al_1, \al_2/2\}$ where $\al_2 := -\log_{\frac{1}{\eta}} \de$ and $\al_1$ is from \cref{claim2.6},  then  for any $0<\rho<R$,  with $(\nabla u)_{\rho} := \fiint_{Q_{\rho}^{\mu_{n_0}}(x_0,t_0)} \nabla u \ dz$, there holds
 \begin{equation*}%
     |\nabla u(x_0,t_0) - (\nabla u)_{\rho}| \leq C \mu \lbr \frac{\rho}{R}\rbr^{\al_3}.
 \end{equation*}%
    \descitem{C5}{conc5} With $\al_3$ as defined in \descref{conc4}{C4}, for any $0<\rho\leq R$ with $(\nabla u)_{\rho} := \fiint_{Q_{\rho}^{\mu_{n_0}}(x_0,t_0)} \nabla u \ dz$, we also have 
     \begin{equation*}%
         \fiint_{Q_{\rho}^{\mu_{n_0}}} |\nabla u - (\nabla u)_{\rho}|^2 \ dz \leq C \mu^2 \lbr \frac{\rho}{R}\rbr^{2\al_3}.
     \end{equation*}%
\end{description}
\end{lemma}

\subsection{Proof of gradient H\"older continuity}
The proof of gradient H\"older continuity follows the proof in~\cite{adimurthiUnifiedApproachAlpha2020} verbatim, which is a delicate adaptation of the covering argument developed by E. DiBenedetto and A. Friedman~\cite{dibenedettoRegularitySolutionsNonlinear1984}. We refer to~\cite{adimurthiUnifiedApproachAlpha2020} for the details, following which the result is proved:

\begin{theorem}
    \label{them1muless1}
        Let $1 < p \leq \{q_1,q_2,\ldots,q_k\}< \infty$  for some $k \in \NN$ and $u$ be a weak solution of the prototype equation
        \[
            u_t - \dv \lbr \mathcal{A}_p(\nabla u) + \sum_{i=1}^k a_i(t) \mathcal{A}_{q_i}(\nabla u)\rbr = 0,
        \]
    where $0 \leq a_i(t) \leq M$ is bounded, measurable functions and $\mathcal{A}_p$ and $\mathcal{A}_{q_i}$ satisfy the growth conditions in~\descref{H1}{H1},~\descref{H2}{H2} for $i\in\{1,2,\ldots,k\}$. Furthermore, assume that \cref{hyp_H}  holds and $|\nabla u| \in L^{\infty}_{\loc}$ holds and            \[ 
        \mu_0 \coloneqq \sup_{4Q_0} |\nabla u|\leq 1.
    \]
    Then given any cylinder $Q_0 = B_{R_0} \times (-R_0^2,R_0^2)$,  there exists $\al= \al(N,p,\{q_i\},C_0,C_1,M) \in (0,1)$ such that for any $z_0, z_1 \in Q_0$, there holds
        \[
            |\nabla u(z_0) - \nabla u(z_1)| \leq C \mu_0^a \lbr \frac{d(z_0,z_1)}{R_0}\rbr^{\al},
        \]
        where $C= C(N,p,\{q_i\},C_0,C_1,M)$, $a= a(N,p,\{q_i\},\al)$.
    \end{theorem}

\subsection{Proof of first alternative - \texorpdfstring{\cref{alt1}}.}
Since $u$ is a weak solution of \cref{main_holder} on $Q_r^{\mu}$ for some $r \in (0,R]$, let us perform the following rescaling: Define
\begin{equation}\label{w_rescale}
        w(x,t) = \frac{u(\mu^{-1}rx,\mu^{-p}r^2t)}{\mu^{-1}r},
    \end{equation}
    then $w$ solves
    \begin{equation}\label{def_w}
        \mu^{p-1} w_t - \mu \dv \lbr \aa_p(\nabla w) +a(t)\aa_q(\nabla w) \rbr = 0 \txt{on} Q_1 := B_1 \times I_1.
    \end{equation}
Let us first prove an energy estimate satisfied by \cref{def_w}.
\begin{lemma}
    \label{energy_w} Let $k \in \RR$ and  $\phi \in C^{\infty}(Q_1)$ be any cut-off function with $\phi = 0$ on $\pa_p(Q_1)$, then the following estimate holds:
    \begin{equation*}
        \begin{array}{l}
            \sup_{t \in I_1}\mu^{p-1} \int_{B_1} (w_{x_i} - k)_-^2 \phi^2  \ dx+ {C_0}\mu \iint_{Q_1}|\nabla w|^{p-2} |\nabla (w_{x_i} - k)_-|^2 \phi^2 \ dz \\
            \hspace*{3cm}\qquad \leq \ 2\mu^{p-1} \iint_{Q_1} (w_{x_i} - k)_-^2 \phi \phi_t \ dz \\ 
            \hspace*{3cm}\qquad \qquad  + \frac{4C_1^2}{C_0}\mu  \iint_{Q_1} \lbr |\nabla w|^p+ M|\nabla w|^q\rbr |\nabla  \phi|^2 \lsb{\chi}{\{w_{x_i} \leq k\}}  \ dz   \\
            \hspace*{3cm}\qquad \qquad + 2C_1\mu \iint_{Q_1} \lbr |\nabla w|^{p-1} + M|\nabla w|^{q-1}\rbr (w_{x_i} -k)_- (|\phi||\nabla^2\phi| + |\nabla \phi|^2)\ dz.
        \end{array}
    \end{equation*}
\end{lemma}
\begin{proof}
    Let us differentiate \cref{def_w} with respect to $x_i$ for some $i \in\{1,2,\ldots,N\}$ and then take $(w_{x_i}-k)_-\phi^2$ with $\phi \in C_c^{\infty}$  as a test function to get
    \begin{eqnarray*}
            &\underbrace{\mu^{p-1}\iint_{Q_1}(w_{x_i})_t (w_{x_i}-k)_-\phi^2\ dz}_{I}+ \underbrace{\mu \iint_{Q_1} \iprod{\partial_i \aa_p(\nabla w)}{\nabla \phi^2} (w_{x_i}-k)_-\ dz}_{II} + \underbrace{\mu \iint_{Q_1} \iprod{\partial_i \aa_p(\nabla w)}{\nabla (w_{x_i}-k)_-} \phi^2 \ dz}_{III} \\
            &\qquad + \underbrace{\mu \iint_{Q_1} a(t) \iprod{\partial_i \aa_q(\nabla w)}{\nabla \phi^2} (w_{x_i}-k)_-\ dz}_{IV} + \underbrace{\mu \iint_{Q_1} a(t) \iprod{\partial_i \aa_q(\nabla w)}{\nabla (w_{x_i}-k)_-} \phi^2 \ dz}_{V} = 0,
    \end{eqnarray*}
    where we have used the notation $\pa_i \aa_{\cdot}(\nabla w) = \aa'_{\cdot}(\nabla w) \nabla w_{x_i}$.  Let us estimate each of the terms as follows:
\begin{description}
    \item[Estimate for $I$:] This can be estimated as follows:
    \begin{equation}\label{est_I}
        \mu^{p-1} \iint_{Q_1} (w_{x_i})_t (w_{x_i}-k)_-\phi^2 = \frac{\mu^{p-1}}{2} \iint_{Q_1} \lbr(w_{x_i}-k)_-^2\phi^2\rbr_t \ dz - {\mu^{p-1}} \iint_{Q_1}  (w_{x_i}-k)_-^2 \phi \phi_t \ dz.
    \end{equation}
    \item[Estimate for $II$:] Integrating by parts, we get
    \begin{equation}\label{est_II}
        \begin{array}{rcl}
            \mu \iint_{Q_1} \iprod{\partial \aa_p(\nabla w)}{\nabla \phi} \phi (w_{x_i}-k)_- \ dz & = & - \mu \iint_{Q_1} \iprod{\aa_p(\nabla w)}{\pa_{x_i} \nabla \phi^2}  \phi (w_{x_i}-k)_- \ dz\\
            && -\mu \iint_{Q_1} \iprod{ \aa_p(\nabla w)}{\nabla \phi^2} \pa_{x_i} (w_{x_i}-k)_- \ dz \\
            & \overset{\redlabel{9.11a}{a}}{\leq} & C_1\mu  \iint_{Q_1} |\nabla w|^{p-1} (|\nabla ^2\phi| + |\nabla \phi|^2) |(w_{x_i}-k)_-|\ dz \\
            && + C_1\mu\ve \iint_{Q_1} |\nabla w|^{p-2} |\phi|^2 |\nabla  (w_{x_i}-k)_-|^2 \ dz\\
            &&+ C_1\frac{\mu}{\ve} \iint_{Q_1} |\nabla w|^{p}|\nabla \phi|^2 \ dz,
        \end{array}
    \end{equation}
    where to obtain \redref{9.11a}{a}, we made use of \descref{H2}{H2}, the trivial bound $|\phi| \leq 1$ and   Young's inequality.
    \item[Estimate for $III$:] We estimate this term as follows:
    \begin{equation}\label{est_III}
            \mu \iint_{Q_1} \iprod{\partial_i \aa_p(\nabla w)}{\nabla (w_{x_i}-k)_-} \phi^2 \ dz  \overset{\descref{H1}{H1}}{\geq} C_0\mu \iint_{Q_1} |\nabla w|^{p-2} |\nabla (w_{x_i}-k)_-|^2 \phi^2 \ dz.
    \end{equation}
    \item[Estimate for $IV$:] This estimate is obtained similarly to~\cref{est_II}:
    \begin{equation}\label{est_IV}
        \begin{array}{rcl}
            \mu \iint_{Q_1} a(t) \iprod{\partial \aa_q(\nabla w)}{\nabla \phi} \phi (w_{x_i}-k)_- \ dz & \leq & C_1M\mu  \iint_{Q_1} |\nabla w|^{q-1} (|\nabla ^2\phi| + |\nabla \phi|^2) |(w_{x_i}-k)_-|\ dz \\
            && + C_1M\mu\ve \iint_{Q_1} |\nabla w|^{q-2} |\phi|^2 |\nabla  (w_{x_i}-k)_-|^2 \ dz\\
            &&+ C_1M\frac{\mu}{\ve} \iint_{Q_1} |\nabla w|^{q}|\nabla \phi|^2 \ dz.
        \end{array}
    \end{equation}
    \item[Estimate for $V$:] We estimate this term in a similar manner to~\cref{est_III}:
    \begin{equation*}
            \mu \iint_{Q_1} a(t) \iprod{\partial_i \aa_q(\nabla w)}{\nabla (w_{x_i}-k)_-} \phi^2 \ dz  \overset{\descref{H1}{H1}}{\geq} C_0\mu \iint_{Q_1} a(t) |\nabla w|^{q-2} |\nabla (w_{x_i}-k)_-|^2 \phi^2 \ dz.
    \end{equation*} However, this term is dropped since $a(t)\geq 0$.
\end{description}
Combining \cref{est_I}, \cref{est_II}, \cref{est_III}, and \cref{est_IV} gives the desired estimate.

\end{proof}

\subsubsection{DeGiorgi type iteration for \texorpdfstring{$u_{x_i}$}.}
The main proposition we prove is the following:
\begin{proposition}\label{prop3.1}
    Let $r \in (0,R]$ and assume that 
    \begin{equation}\label{3.3}
        \sup_{Q_r^{\mu}} \|\nabla u\| \leq A \mu
    \end{equation}
    holds for some $A \geq 1$. For any $i \in \{1,2,\ldots,N\}$,  there exists universal constant $\nu \in (0,1/2)$ such that if
    \begin{equation*}%
        | \{ (x,t) \in Q_r^{\mu} : u_{x_i} < \mu/2\}| \leq \nu |Q_r^{\mu}|,
    \end{equation*}%
    holds, then we have the following conclusion:
    \begin{equation*}%
        u_{x_i} \geq \frac{\mu}{4} \txt{on} Q_{r/2}^{\mu}.
    \end{equation*}%
\end{proposition}
Following the rescaling from \cref{w_rescale}, we can restate the following equivalent version of \cref{prop3.1} for $w$ as follows:
\begin{lemma}\label{lemma3.1}
    Suppose 
    \begin{equation}\label{3.3_w}
        \sup_{Q_1} \|\nabla w\| \leq A \mu
    \end{equation}
    holds for some $A \geq 1$. For any $i \in \{1,2,\ldots,N\}$,  there exists universal constant $\nu \in (0,1/2)$ such that if
    \begin{equation}\label{9.48}
        | \{ (x,t) \in Q_1 : w_{x_i} < \mu/2\}| \leq \nu |Q_1|,
    \end{equation}
    holds, then we have the following conclusion:
    \begin{equation*}
        w_{x_i} \geq \frac{\mu}{4} \txt{on} Q_{1/2}.
    \end{equation*}
\end{lemma}
\begin{proof}
Let us define the following constants:
\begin{equation}\label{def_k_m}
        k_m:= k_0 - \frac{H}{8(1+A)} \lbr 1 - \frac{1}{2^m}\rbr \txt{where} H:= \sup_{Q_1} (w_{x_i}-k_0)_- \txt{and} k_0:= \frac{\mu}{2}.
    \end{equation}%
    Let us also define the following sequence of dyadic parabolic cylinders:
    \begin{equation*}
        Q_m:= Q_{\rho_m} \txt{where} \rho_m := \frac12 + \frac{1}{2^{m+1}}.
    \end{equation*}
    Note that $Q_m \rightarrow Q_{1/2}$ and $Q_0 = Q_1$. 
    Furthermore, let us consider the following cut-off function $\eta_m \in C^{\infty}(Q_m)$ with $\eta_m = 0$ on $\pa_pQ_m$ and $\eta_m \equiv 1$ on $ Q_{m+1}$. The following bounds hold,
    \begin{equation*}
        |\nabla^ 2\eta_m| + |\nabla\eta_m|^2 + |(\eta_m)_t| \leq C(N) 4^m.
    \end{equation*}%
    We will split the proof of the lemma into several steps.
    \begin{description}[leftmargin=*]
        \ditem{\underline{Step $1$:}}{step1} In this step, we show that without loss of generality, we can assume $4H \geq \mu$. Suppose not, then we would have $4H < \mu$ which implies
        \begin{equation*}
        \sup_{Q_1} (w_{x_i}-k_0)_- = \frac{\mu}{2} - \inf_{Q_1} w_{x_i} < \frac{\mu}{4}.
    \end{equation*}
    This says  $w_{x_i} > \frac{\mu}{4}$ and the desired conclusion of \cref{lemma3.1}  follows.
    With the assumption on $H$, the choice of $k_m$ in~\eqref{def_k_m} satisfy the following bounds:
    \begin{equation}\label{3.18}
        k_m - k_{m+1} \geq \frac{\mu}{2^{m+6}(1+A)}, \qquad k_m \geq \frac{\mu}{4} \txt{and} k_m \rightarrow k_{\infty} = k_0 - \frac{H}{8(1+A)} > \frac{\mu}{4}.
    \end{equation}
        \ditem{\underline{Step $2$:}}{step2} In this step, we will apply \cref{sobolev-poincare} to obtain an estimate for the level sets. In order to do this, let us define 
        \begin{equation*}%\label{def_A_m}
            A_m := \{(x,t) \in Q_m: w_{x_i} < k_m\},
    \end{equation*}%
    and consider the following function
    \begin{equation*}%\label{def_w_m}
        \tw_m := \left\{ \begin{array}{ll}
                            0 &\ \  \text{if} \ \ w_{x_i} > k_m, \\
                            k_m - w_{x_i} &\ \  \text{if} \ \ k_m \geq w_{x_i} > k_{m+1}, \\
                            k_m - k_{m+1} &\ \  \text{if} \ \  k_{m+1} \geq w_{x_i}. 
                        \end{array}\right.
    \end{equation*}%
    Since $\eta_m = 1$ on $Q_{m+1}$, we obtain the following sequence of estimates: 
    \begin{equation}\label{3.21}
        \begin{array}{rcl}
            \mu^{p-1} (k_m - k_{m+1})^2 |A_{m+1}| & = & \mu^{p-1} \|\tw_m\|_{L^2(A_{m+1})}^2\\
            & \overset{\redlabel{321a}{a}}{\leq} & \mu^{p-1} \|\tw_m\eta_m\|_{L^2(Q_{m})}^2\\
            & \overset{\redlabel{321b}{b}}{\apprle} &  \mu^{p-1} \|\tw_m\eta_m\|_{V^{2,2}(Q_{m})}^2 |A_m|^{\frac{2}{N+2}},
        \end{array}
    \end{equation}
    where to obtain \redref{321a}{a}, we enlarged the domain and made use of the fact that $\spt(\eta_m) \subset Q_m$ and to obtain \redref{321b}{b}, we made use of \cref{sobolev-poincare} noting that $\tw_m \eta_m$ is non-negative.
    Then, observing that 
    \[
        \tw_m \leq (w_{x_i} - k_m)_- \txt{and} |\nabla \tw_m| \leq |\nabla (w_{x_i} - k_m)_-|  \lsb{\chi}{Q_1 \setminus \{w_{x_i} < k_{m+1}\}},
    \]
    and recalling \cref{func_space},  we get
    \begin{equation}\label{9.55}
        \begin{array}{rcl}
            \mu^{p-1} \|\tw_m\eta_m\|_{V^2(Q_{m})}^2 & \leq & \sup_{-1<t<1} \mu^{p-1} \int_{B_1} (w_{x_i}-k_m)_-^2 \eta_m^2 \ dx \\
            && \qquad + \mu^{p-1} \iint_{Q_1} |\nabla (w_{x_i} - k_m)_-|^2 \lsb{\chi}{Q_1 \setminus \{w_{x_i} < k_{m+1}\}} \eta_m^2 \ dz\\
            &&\qquad  + \mu^{p-1} \iint_{Q_1} (w_{x_i} - k_m)_-^2 |\nabla \eta_m|^2\ dz.
        \end{array}
    \end{equation}
        \ditem{\underline{Step $3$:}}{step3} In this step, we shall estimate \cref{9.55} and obtain a suitable decay of the level set $A_m$ as follows: From \cref{3.18}, we see that 
        \begin{equation}\label{9.56}
            \mu \leq 4 k_{m+1} \leq 4 w_{x_i} \leq 4|\nabla w| \txt{on} Q_1 \setminus \{w_{x_i} < k_{m+1}\}.
        \end{equation}
        Thus making use of \cref{9.56} into \cref{9.55}, we get
    \begin{equation}\label{3.23}
        \begin{array}{rcl}
            \mu^{p-1} \|\tw_m\eta_m\|_{V^2(Q_{m})}^2 & \leq & \sup_{-1<t<1} \mu^{p-1} \int_{B_1} (w_{x_i}-k_m)_-^2 \eta_m^2 \ dx \\
            && \qquad +  4^{p-1}\iint_{Q_1} |\nabla w|^{p-1} |\nabla (w_{x_i} - k_m)_-|^2 \lsb{\chi}{Q_1 \setminus \{w_{x_i} < k_{m+1}\}} \eta_m^2 \ dz\\
            &&\qquad  + \mu^{p-1} \iint_{Q_1} (w_{x_i} - k_m)_-^2 |\nabla \eta_m|^2\ dz.
        \end{array}
    \end{equation}
    From \cref{3.3_w}, we see that 
    \begin{equation}\label{3.24}    
            A \mu \iint_{Q_1} |\nabla w|^{p-2} |\nabla (w_{x_i}-k)_-|^2 \phi^2 \ dz \geq  \iint_{Q_1} |\nabla w|^{p-1} |\nabla (w_{x_i}-k)_-|^2 \phi^2\ dz.
    \end{equation}
    Thus substituting \cref{3.24} into \cref{3.23} and making use of \cref{energy_w} to estimate each of the terms appearing on the right hand side of \cref{3.23} using \cref{3.3_w}, we get
\begin{equation}\label{3.25}
        \begin{array}{rcl}
            \mu^{p-1} \|\tw_m\eta_m\|_{V^2(Q_{m})}^2 & \lesssim & \mu \iint_{Q_1}(|\nabla w|^p+|\nabla w|^q) |\nabla \eta_m|^2\lsb{\chi}{\{w_{x_i} \leq k_m\}}\ dz\\
            && + \mu \iint_{Q_1} (|\nabla w|^{p-1}+|\nabla w|^{q-1}) (w_{x_i} - k_m)_- ( |\eta_m||\nabla^2\eta_m| + |\nabla \eta_m|^2)\\
            && + \mu^{p-1} \iint_{Q_1} (w_{x_i} - k_m)_-^2  |\nabla  \eta_m|^2 \\
            & \leq & C \mu^{p+1} (1+\mu^{\frac{q}{p}}+\mu^{\frac{q-1}{p-1}}) 4^m  |A_m|\\
            & \leq & C \mu^{p+1} 4^m  |A_m|,
\end{array}
\end{equation} where the last inequality follows from $\mu\leq \mu_0\leq 1$ and by redefining constant $C$.
Thus combining \cref{3.25} with \cref{3.21} and making use of \cref{3.18}, we get
\begin{equation*}
    \mu^{p-1}\frac{\mu^2}{4^{m+6} (1+A)^2} |A_{m+1}| \leq C 4^m\mu^{p+1} |A_m|^{1+\frac{2}{N+2}} \Longleftrightarrow |A_{m+1}| \leq C16^m  |A_m|^{1+\frac{2}{N+2}}.
\end{equation*}%
We can now apply \cref{iteration} to obtain the existence of a universal constant $\nu \in (0,1)$ such that if
\[
    |A_0|= |\{Q_1: w_{x_i} < \mu/2\}|  < \nu |Q_1|,
\]
then $|A_m| \rightarrow 0$ as $m \rightarrow \infty$. In particular, this says
\[
    w_{x_i} \geq \frac{\mu}{4} \txt{on} Q_{1/2}.
\]
    \end{description}
This completes the proof of the lemma. 
\end{proof}

Following analogous calculations, we can also obtain the dual version of \cref{prop3.1}.
\begin{proposition}\label{prop_upper_bnd}
    Let $r \in (0,R]$ and assume that 
    \begin{equation*}
        \sup_{Q_r^{\mu}} \|\nabla u\| \leq A \mu
    \end{equation*}
    holds for some $A \geq 1$. For any $i \in \{1,2,\ldots,N\}$,  there exists universal constant $\nu \in (0,1/2)$ such that if
    \begin{equation}\label{9.48_alt}
        | \{ (x,t) \in Q_r^{\mu} : u_{x_i} > -\mu/2\}| \leq \nu |Q_r^{\mu}|,
    \end{equation}
    holds, then we have the following conclusion:
    \begin{equation*}
        u_{x_i} \leq -\frac{\mu}{4} \txt{on} Q_{r/2}^{\mu}.
    \end{equation*}
\end{proposition}

\subsubsection{Proof of the decay estimate from applying the linear theory}
Let us first recall the weak Harnack inequality and a decay estimate that will be needed to complete the proof of \cref{alt1}, the details can be found in \cite[Lemma 3.1]{kuusiWolffGradientBound2014}.
\begin{lemma}\label{lemma_linear}
    Let $v \in L^2(-1,1;W^{1,2}(B_1))$ be a weak solution to the linear parabolic equation 
    \[
        v_t - \dv (B(x,t) \nabla v) = 0,
    \]
    where the matrix $B(x,t)$ is bounded, measurable and satisfies
    \[
        C_0 |\zeta|^2 \leq \iprod{B(x,t)\zeta}{\zeta} \txt{and} |B(x,t)| \leq C_1,
    \]
    for any $\zeta \in \RR^N$ and $0<C_0 \leq C_1$ are fixed constants. Then there exists a constant $C = C(N,C_0,C_1) \geq 1$ and $\be = \be(N,C_0,C_1) \in (0,1)$ such that the following estimates are satisfied:
    \begin{gather*}
        \sup_{Q_{1/2}}|v| \leq C \lbr \fiint_{Q_1} |v|^q \ dz \rbr^{\frac{1}{q}} \txt{for any $q \in [1,2]$,}\\
        \lbr \fiint_{Q_{\de}} |v - \avgs{v}{Q_{\de}}|^q \ dz \rbr^{\frac{1}{q}} \leq C \de^{\be} \lbr \fiint_{Q_{1}} |v - \avgs{v}{Q_{1}}|^q \ dz \rbr^{\frac{1}{q}} \txt{whenever $q \in [1,2]$ and $\de \in (0,1)$.}
    \end{gather*}
\end{lemma}
 We have now collected all the estimates that are required in the proof of \cref{alt1}. The following proof follows closely that of \cite[Lemma 3.2]{kuusiWolffGradientBound2014} and we only verify the hypothesis on the linearised equation.
 \begin{lemma}\label{lemma3.2}
     Assume that in the cylinder $Q_r^{\mu}$, the following is satisfied for a fixed $A \geq 1$:
     \begin{equation*}%
         0 < \frac{\mu}{4} \leq \|\nabla u\|_{L^{\infty}(Q_r^{\mu})} \leq A \mu.
     \end{equation*}%
     Then there exists constants $\be= \be(N,p,C_0,C_1,M,A) \in (0,1)$ and $C = C(N,p,C_0,C_1,M,A) \geq 1$ such that the following holds for any $\de \in (0,1)$ and $q \geq 1$:
     \begin{equation*}%
         \lbr \fiint_{Q_{\de r}^{\mu}} |\nabla u - \avgs{\nabla u}{Q_{\de r}^{\mu}}|^q \ dz \rbr^{\frac{1}{q}} \leq C \de^{\be} \lbr \fiint_{Q_{r}^{\mu}} |\nabla u - \avgs{\nabla u}{Q_{r}^{\mu}}|^q \ dz \rbr^{\frac{1}{q}}.
     \end{equation*}%
 \end{lemma}
 \begin{proof}
     We rescale $u$ according to \cref{w_rescale} as was done in the proof of \cref{lemma3.1}, then the hypothesis becomes 
     \begin{equation}\label{first_alt_hyp}
         0 < \frac{\mu}{4} \leq \|\nabla w\|_{L^{\infty}(Q_1)} \leq A \mu.
     \end{equation}
Differentiating \cref{def_w} and dividing the resulting equation by $\mu^{p-1}$, we see that $w_{x_i}$ solves
\begin{equation*}
    (w_{x_i})_t - \dv (B(x,t) \nabla w_{x_i}) = 0 \txt{where} B(x,t) = \mu^{2-p} \lbr \pa \aa_p(\nabla w)+a(t) \pa \aa_q(\nabla w) \rbr.
\end{equation*}%
Making use of \descref{H1}{H1} and \descref{H2}{H2} along with \cref{first_alt_hyp}, we see that the matrix $B(x,t)$ satisfies the following bounds:
\begin{equation*}%\label{w_x_i_sol}
    C |\zeta|^2 \leq \iprod{B(x,t)\zeta}{\zeta} \leq C\left( \lbr \frac{\mu}{\mu} \rbr^{p-2} + \lbr \frac{\mu^{q-2}}{\mu^{p-2}} \rbr \right) |\zeta|^2 = C\left( \lbr \frac{\mu}{\mu} \rbr^{p-2} + \lbr \frac{\mu^{q-1}}{\mu^{p-1}} \rbr \right) |\zeta|^2 \leq C |\zeta|^2,
\end{equation*}
for any $\zeta \in \RR^N$ and $C= C(N,p,C_0,C_1,M,A)$. The last inequality follows due to $\mu\leq \mu_0\leq 1$. 
\end{proof}

 \subsubsection{Proof of \texorpdfstring{\cref{alt1}}.}
 The proof of the following proposition follows verbatim as in \cite[Proposition 3.3]{kuusiWolffGradientBound2014} by applying the estimates obtained in the previous subsections.
 \begin{proposition}
     \label{prop3.3}
     Assume that \cref{3.3} is in force, then there exists $\nu = \nu(N,p,C_0,C_1,A)\in (0,1/2)$ such that if there exists $i \in \{1,2,\ldots,N\}$ such that either \cref{9.48} or \cref{9.48_alt} holds, then there exists $\be = \be(N,p,C_0,C_1,A) \in (0,1)$ and $C= C(N,p,C_0,C_1,A)$ such that for any $\de \in (0,1)$, the following conclusions hold:
     \begin{gather*}
                  \lbr \fiint_{Q_{\de r}^{\mu}} |\nabla u - \avgs{\nabla u}{Q_{\de r}^{\mu}}|^q \ dz \rbr^{\frac{1}{q}} \leq C \de^{\be} \lbr \fiint_{Q_{r}^{\mu}} |\nabla u - \avgs{\nabla u}{Q_{r}^{\mu}}|^q \ dz \rbr^{\frac{1}{q}},\\
                  \|\nabla u\| \geq \frac{\mu}{4} \txt{on $Q_{r/2}^{\mu}$.}
     \end{gather*}
 \end{proposition}

\subsection{Proof of second alternative - \texorpdfstring{\cref{alt2}}.}

With $\nu$ fixed as in \cref{alt1}, we now address the alternative where \cref{9.48} and \cref{9.48_alt} do not hold. This condition is equivalent to 
\begin{equation}\label{alt2_cond}
        |\{ Q_{r}^{\mu}: u_{x_i} \geq \mu/2\}| < (1-\nu) |Q_r^{\mu}| \txt{and}|\{ Q_{r}^{\mu}: u_{x_i} \leq -\mu/2\}| < (1-\nu) |Q_r^{\mu}|.
\end{equation}
We assume the following is always satisfied for some $r \in (0,R]$:
\begin{equation*}%\label{3.3_alt2}
 \sup_{Q_r^{\mu}} \|\nabla u\| \leq \mu.
\end{equation*}

Let us perform the following change of variables:
    \begin{equation}\label{9.69}
        w(x,t) = \frac{u(\mu^{-1}rx,\mu^{-p}r^2t)}{r} \txt{for} (x,t) \in Q_1.
    \end{equation}
    Then we have 
    \begin{equation}\label{9.70}
        \sup_{Q_1} \|\nabla w\| \leq 1 \txt{on} Q_1.
    \end{equation}
    Moreover, \cref{alt2_cond} becomes
    \begin{equation*}
%         \label{alt2_new}
        \abs{\left\{ Q_{1}: w_{x_i} \geq \frac{1}{2}\right\}} < (1-\nu) |Q_1| \txt{and}\abs{\left\{ Q_{1}: w_{x_i} \leq -\frac{1}{2}\right\}} < (1-\nu) |Q_1|.
    \end{equation*}

\subsubsection{Choosing a good time slice}
First let us show there exists a good time slice.
\begin{lemma}\label{cht}
    There exists $t_{\ast}$ such that $ -1 \leq t_{\ast} \leq -\frac{\nu}{2}$ and
    \begin{equation}\label{9.68}
        \abs{\left\{ B_{1} : w_{x_i}(x,t_{\ast}) \geq \frac{1}{2}\right\}} \leq \lbr \frac{1-\nu}{1-\nu/2}\rbr |B_{1}|
    \end{equation}

\end{lemma}

\begin{proof}
    The proof is by contradiction, suppose \cref{9.68} does not hold for any $t \in (-1,-\frac{\nu}{2})$, then we have 
    \begin{equation*}
        \begin{array}{rcl}
        (1-\nu) |Q_1|\overset{\cref{alt2_cond}}{>}    \abs{\left\{ Q_1 : w_{x_i} \geq \frac{1}{2}\right\}} & = & \int_{-1}^{-\frac{\nu}{2}}  \abs{\left\{ B_{1} : w_{x_i}(x,t) \geq \frac{1}{2}\right\}} \ dt \\
            &\overset{\text{\cref{9.68} fails}}{>} & \lbr (1-\nu)\rbr |B_{1}| ,
        
        \end{array}
    \end{equation*}
     which is a contradiction.
\end{proof}

\subsubsection{Rescaling the equation}
From the rescaling \cref{9.69}, let us define
    \begin{equation}\label{9.73}
        \hat{\aa}(t,\zeta) := \frac{1}{\mu^{p-1}} \lbr \aa_p(\mu \zeta)+a(t)\aa_q(\mu \zeta)\rbr,
    \end{equation}
        then, we see that $w$ solves 
        \begin{equation}\label{9.74}
            w_t - \dv \hat{\aa}(t,\nabla w) = 0 \txt{in} Q_1. 
        \end{equation}

\begin{lemma}\label{lemma_elip_resc}
    Differentiating \cref{9.74} with respect to $x_i$ for some $i \in \{1,2,\ldots,N\}$, we get
    \begin{equation}\label{w_x_i}
    (w_{x_i})_t - \dv \pa\hat{\aa}(\nabla w)\nabla w_{x_i} = 0 \txt{in} Q_1. 
    \end{equation}
    Then $\pa\hat{\aa}(\zeta)$ and $\hat{\aa}(\zeta)$ satisfies the following structure conditions:
    \begin{equation}\label{9.75}
        \begin{array}{c}
            |\hat{\aa}(\zeta)| + |\pa\hat{\aa}(\zeta)| |\zeta| \leq C_1   \lbr |\zeta|^{p-1} + M|\zeta|^{q-1}\rbr,\\
            C_0 |\zeta|^{p-2} |\eta|^2 \leq \iprod{\pa \hat{\aa}(\zeta)\eta}{\eta}.
        \end{array}
    \end{equation}
\end{lemma}
\begin{proof}
    The proof follows directly from \descref{H1}{H1},~\descref{H2}{H2}, \cref{9.73}, and the condition $\mu\leq \mu_0\leq 1$.
\end{proof}

\subsubsection{Logarithmic and Energy estimates}
\begin{lemma}\label{log_estimate}
    Let $0 < \eta_0 < \nu$ and $k \geq \frac{1}{4}$ be any two fixed numbers (recall $\nu$ is from \cref{alt2}) and consider the function
\begin{equation*}%
    \Psi(z) := \log^+ \lbr \frac{\nu}{\nu - (z-(1-\nu))_+ + \eta_0} \rbr.
\end{equation*}%
Then there exists constant $C = C(N,p,C_0,C_1,M)$ such that for all $ t_1,t_2 \in (-1,1)$ with $t_1<t_2$ and all $s \in (0,1)$, there holds
\begin{equation*}%\label{log_est}
    \int_{B_{s} \times \{t_2\}} \Psi^2((w_{x_i}-k)_+) \ dx \leq \int_{B_{1} \times \{t_1\}} \Psi^2((w_{x_i}-k)_+) \ dx + \frac{C}{(1-s)^2} \iint_{B_1\times (t_1,t_2)} \Psi((w_{x_i}-k)_+) \ dz. 
\end{equation*}
\end{lemma}
\begin{proof}
    Let us take $\Psi((w_{x_i}-k)_+) \Psi'((w_{x_i}-k)_+) \zeta^2$ as a test function in \cref{9.74}. Since $k \geq \frac1{4}$, we see that $(w_{x_i} - k)_+ =0 $ whenever $w_{x_i} \leq \frac1{4}$. On this set, $\Psi(0) = \log^+ \lbr \frac{\nu}{\nu+\eta_0}\rbr = 0$, thus we see that $\spt{\Psi((w_{x_i}-k)_+)}$ is contained in  the set $\left\{w_{x_i} \geq \frac1{4}\right\}$. Thus making use of \cref{9.70}, we make the following observations,
    \begin{align}
        \label{ellip_new}
        \frac{1}{4^{p-1}} \leq \frac{ |\nabla w|^{p-1} }{  |\nabla w| } + M\frac{ |\nabla w|^{q-1} }{  |\nabla w| }  \leq 4 + 4M \txt{on} \left\{w_{x_i} \geq \frac1{4}\right\}.
    \end{align}
Substituting \cref{ellip_new} into \cref{9.75}, we see that $\pa \hat{\aa}(\nabla w)$ is uniformly elliptic and we rewrite \cref{w_x_i} as follows:
\begin{equation}\label{unif_elliptic}
    (w_{x_i})_t - \dv B(x,t) \nabla w_{x_i} = 0, \txt{with} \frac{C_0}{4^{p-1}}|\zeta|^2 \leq \iprod{B(x,t) \zeta}{\zeta} \leq {4(1+M)C_1} |\zeta|^2.
\end{equation}
We can now follow the proof of \cite[Proposition 3.2 of Chapter II]{dibenedettoDegenerateParabolicEquations1993} (see also \cite[(12.7) of Chapter IX]{dibenedettoDegenerateParabolicEquations1993}) to get 
\[
    \int_{B_{s} \times \{t_2\}} \Psi^2((w_{x_i}-k)_+) \ dx \leq \int_{B_{s} \times \{t_1\}} \Psi^2((w_{x_i}-k)_+) \ dx + \frac{C_{(N,P,C_0,C_1,M)}}{(1-s)^2} \iint_{B_1\times (t_1,t_2)} \Psi((w_{x_i}-k)_+) \ dz. 
\]

\end{proof}

\begin{lemma}
    \label{energy_estimate}
    Let $k \geq \frac1{4}$ be some fixed constant, then $w$ solving \cref{9.74} satisfies
    \[
    \begin{array}{l}
        \sup_{t_0<t<t_1} \int_{B_1}  (w_{x_i} - k)_+^2 \zeta^2 \ dx  + C \iint_{B_1\times (t_0,t_1)} |\nabla (w_{x_i} - k)_+|^2 \zeta^2 \ dz\leq \int_{B_1 \times \{t=t_0\}} (w_{x_i} - k)_+^2 \zeta^2 \ dx \\
        \hspace*{10cm} + C\iint_{B_1 \times (t_0,t_1)} (w_{x_i} - k)_+^2 |\nabla \zeta|^2  \ dz\\
        \hspace*{10cm} + C\iint_{B_1 \times (t_0,t_1)} (w_{x_i} - k)_+^2 |\zeta| |\zeta_t|  \ dz,
        \end{array}
    \]
where $t_0,t_1 \in [-1,1)$ with $t_0<t_1$ are some fixed time slices and $\zeta \in C^{\infty}(Q_1)$ is a cut-off function. 
\end{lemma}
\begin{proof}
    We take $(w_{x_i} - k)_+ \zeta^2$ as a test function in \cref{9.74}, noting that we retain the uniformly elliptic structure of $B(x,t)$ in \cref{unif_elliptic} from the choice of test function. Thus proceeding according to \cite[Proposition 3.1 of Chapter II]{dibenedettoDegenerateParabolicEquations1993} gives the desired estimate.
\end{proof}

\subsubsection{De Giorgi type iteration}

The proof of the following lemma, showing an expansion of positivity in time, can be found in~\cite[Lemma~8.22]{adimurthiUnifiedApproachAlpha2020}.
\begin{lemma}\label{lemma921}
    There exists $\eta_0 = \eta_0(N,p,\nu) \in (0,\nu)$ such that for all $t \in (t_{\ast},1)$ with $t_{\ast}$ as in Lemma \ref{cht}, there holds
    \[
        \abs{\left\{x \in B_1: w_{x_i} (x,t) > (1-\eta_0)  \right\}} \leq \lbr 1-\frac{\nu^2}{4} \rbr |B_1|.
    \]
\end{lemma}

Given Lemma  \ref{lemma921}, the rest of the proof of the second alternative Proposition \ref{alt2} is as in the linear case. Complete details of the rest of the proof may be found in~\cite{adimurthiUnifiedApproachAlpha2020}.

\section{Proof of \texorpdfstring{\cref{corthem1}}. for $\mu_0>1$.} The proof of~\cref{corthem1} for the case $\mu_0\leq 1$ is given above in~\cref{them1muless1}. For the case $\mu_0>1$, we proceed as follows. We define the following rescaling:
\[
    \tilde{w}(x,t)=\frac{u(x,t)}{\mu_0},    
\]
then $\tilde{w}$ satisfies the equation
\begin{equation*}
    \tilde{w}_t-\dv \lbr \tilde{\aa}_p(\nabla \tilde{w}) + \sum_{i=1}^k a_i(t) \tilde{\aa}_{q_i}(\nabla \tilde{w}) \rbr = 0,
\end{equation*} where 
\[
     \tilde{\aa}_p(\zeta)=\frac{1}{\mu_0}\aa_p(\mu_0\zeta)\txt{and} \tilde{\aa}_{q_i}(\zeta)=\frac{1}{\mu_0}\aa_{q_i}(\mu_0\zeta),
\] and the following structure conditions are satisfied for $i\in\{1,2,\ldots,k\}$:
\begin{description}
    \descitem{S1}{S1} $\left\langle \tilde{\mathcal{A}}'_p(z)\zeta,\zeta\right\rangle\geq C_0 \mu_0^{p-2} |z|^{p-2}|\zeta^2|$ and $\left\langle \tilde{\mathcal{A}}'_{q_i}(z)\zeta,\zeta\right\rangle\geq C_0 \mu_0^{q_i-2} |z|^{q_i-2}|\zeta^2|$.
    \descitem{S2}{S2} $|\tilde{\mathcal{A}}_p(z)|+|\tilde{\mathcal{A}}'_p(z)||z| \leq C_1 \mu_0^{p-2} |z|^{p-1}$ and $|\tilde{\mathcal{A}}_{q_i}(z)|+|\tilde{\mathcal{A}}'_{q_i}(z)||z| \leq C_1 \mu_0^{q_i-2} |z|^{q_i-1}$.
\end{description}
Since $\tilde{w}$ satisfies $\sup_{4Q_0}|\nabla\tilde{w}|\leq 1$, we are in the regime of~\cref{them1muless1}, whose application gives~\cref{corthem1}.

\section{Proof of \texorpdfstring{\cref{corthem2}}.} 
We begin by proving~\cref{corthem2} for the case $\mu_0\leq 1$. The proof is similar to that of~\cref{them1muless1} in all respects except in the proof of~\cref{alt1} and~\cref{alt2}. We shall list below the modifications required in the variable exponent case:
Let $u$ be a weak solution of 
\begin{equation}
    \label{main_holder_cor2}
    u_t - \dv \lbr |\nabla u|^{p(t)-2} \nabla u\rbr= 0 \txt{on} 4Q_0,
\end{equation}
with $1 < p \leq p(t) \leq q < \infty$. 
\begin{description}
    \item[Modification 1:] The analogue of \cref{energy_w} for the double phase takes the following form:  Since $u$ is a weak solution of \cref{main_holder_cor2} on $Q_r^{\mu}$ for some $r \in (0,R]$, let us perform the following rescaling: Define
\begin{equation}\label{w_rescale_cor2}
        w(x,t) = \frac{u(\mu^{-1}rx,\mu^{-p}r^2t)}{\mu^{-1}r},
    \end{equation}
    then $w$ solves
    \begin{equation}\label{def_w_cor2}
        \mu^{p-1} w_t - \mu \dv \aa(\nabla w) = 0 \txt{on} Q_1 := B_1 \times I_1.
    \end{equation}
  Then the following energy estimate is satisfied:  Let $k \in \RR$ and  $\phi \in C^{\infty}(Q_1)$ be any cut-off function with $\phi = 0$ on $\pa_p(Q_1)$ then the following  holds:
\begin{equation*}
            \begin{array}{l}
                \sup_{t \in I_1}\mu^{p-1} \int_{B_1} (w_{x_i} - k)_-^2 \phi^2  \ dx+ {C_0}\mu \iint_{Q_1}|\nabla w|^{p(t)-2} |\nabla (w_{x_i} - k)_-|^2 \phi^2 \ dz \\
                \hspace*{3cm}\qquad \leq \  \mu^{p-1} \int_{B_1\times \{t=1\}} (w_{x_i} - k)_-^2 \phi^2  \ dx+ 2\mu^{p-1} \iint_{Q_1} (w_{x_i} - k)_-^2 \phi \phi_t \ dz \\
                \hspace*{3cm}\qquad \qquad  + \frac{4C_1^2}{C_0}\mu  \iint_{Q_1}  |\nabla w|^{p(t)} |\nabla  \phi|^2 \lsb{\chi}{\{w_{x_i} \leq k\}}  \ dz   \\
                \hspace*{3cm}\qquad \qquad + 2C_1\mu \iint_{Q_1}  |\nabla w|^{p(t)-1} (w_{x_i} -k)_- (|\phi||\nabla^2\phi| + |\nabla \phi|^2)\ dz.
            \end{array}
        \end{equation*}
    \item[Modification 2:] The analogue of \cref{3.25} becomes
\begin{equation*}
        \begin{array}{rcl}
            \mu^{p-1} \|\tw_m\eta_m\|_{V^2(Q_{m})}^2 & \leq & \mu^{p-1} \iint_{Q_1} (w_{x_i} - k_m)_-^2  \eta_m(\eta_m)_t^2\ dz \\
            && +  \mu \iint_{Q_1}|\nabla w|^{p(t)}   |\nabla \eta_m|^2\lsb{\chi}{\{w_{x_i} \leq k_m\}}\ dz\\
            && + \mu \iint_{Q_1}|\nabla w|^{p(t)-1} (w_{x_i} - k_m)_- ( |\eta_m||\nabla^2\eta_m| + |\nabla \eta_m|^2)\\
            && + \mu^{p-1} \iint_{Q_1} (w_{x_i} - k_m)_-^2  |\nabla  \eta_m|^2 \\
            & \leq & C \mu^{p+1} 4^m  |A_m|,
\end{array}
\end{equation*} where the last inequality follows from $\mu\leq \mu_0\leq 1$.%
    \item[Modification 3:] It is easy to see that the analogue of the matrix $B(x,t)$ in \cref{lemma3.2} is also uniformly elliptic. As in the proof of \cref{lemma3.1}, we rescale according to \cref{w_rescale_cor2}, then the hypothesis becomes 
     \begin{equation*}
         0 < \frac{\mu}{4} \leq |\nabla w(x,t)| \leq  A\mu \txt{for any} (x,t) \in Q_1.
     \end{equation*}%
Differentiating \cref{def_w_cor2} and dividing the resulting equation by $\mu^{p-1}$, we see that $w_{x_i}$ solves
\begin{equation*}
    (w_{x_i})_t - \dv (B(x,t) \nabla w_{x_i}) = 0 \txt{where} B(x,t) = \mu^{2-p} \pa \aa(\nabla w).
\end{equation*}%
Furthermore, we see that the matrix $B(x,t)$ satisfies the following bounds:
\[
    C |\zeta|^2 \leq \iprod{B(x,t)\zeta}{\zeta} \leq C  |\zeta|^2,
\]
for any $\zeta \in \RR^N$ and $C= C(N,p,q,C_0,C_1)$.

    \item[Modification 4:]  Let us perform the following change of variables for the analogue of \cref{9.69}
    \begin{equation*}
        w(x,t) = \frac{u(\mu^{-1}rx,\mu^{-p}r^2t)}{r} \txt{for} (x,t) \in Q_1.
    \end{equation*}%
    Then we have 
    \begin{equation*}
       |\nabla w(x,t)| \leq 1 \txt{for any} (x,t) \in Q_1.
    \end{equation*}%
    The analogue of \cref{ellip_new} in \cref{log_estimate} now becomes
    \begin{equation*}
        \frac{1}{4^{q-1}}\leq\frac{1}{4^{p(t)-1}} \leq    \frac{|\nabla w|^{p(t)-1}}{|\nabla w|}   \leq 4 \txt{on} \left\{w_{x_i} \geq \frac1{4}\right\}.
    \end{equation*}
\end{description}
This concludes the adaptation of \cref{them1muless1} to obtain \cref{corthem2} in the case $\mu_0\leq 1$. 

\subsection{Proof of \texorpdfstring{\cref{corthem2}}. for $\mu_0>1$.} 

For the case $\mu_0>1$, we proceed as follows. We define the following rescaling:
\[
    \tilde{w}(x,t)=\frac{u(x,t)}{\mu_0},    
\]
then $\tilde{w}$ satisfies the equation
\begin{equation*}
    \tilde{w}_t-b(t)\dv \lbr |\nabla \tilde{w}|^{p(t)-2}\nabla \tilde{w} \rbr = 0 \txt{where}b(t)=\mu_0^{p(t)-2}.
\end{equation*} 
The function $b(t)$ satisfies $\mu_0^{p-2}\leq b(t)\leq \mu_0^{q-2}$. Since $\tilde{w}$ satisfies $\sup_{4Q_0}|\nabla\tilde{w}|\leq 1$, we can apply the previous case.

\bibliography{gradientholdermultiphase}
\bibliographystyle{plainurl}

\end{document}